\documentclass{amsart}

\usepackage[utf8]{inputenc}

\usepackage{xcolor}
\usepackage{tikz}
\usetikzlibrary{arrows.meta}
\usetikzlibrary{calc}

\pgfmathsetseed{3}

\usepackage{amssymb}
\usepackage{amsmath}
\usepackage{amsfonts}
\usepackage{tikz-cd}
\usetikzlibrary{arrows}
\usepackage{relsize}
\usepackage{thmtools, thm-restate}

\usepackage{pictex}

\usepackage{graphicx}
\graphicspath{ {./images/} }
\usepackage{wrapfig}
\usepackage[font=footnotesize]{caption}

\tikzset{
curvarr/.style={
  to path={ -- ([xshift=2ex]\tikztostart.east)
    |- (#1) [near end]\tikztonodes
    -| ([xshift=-2ex]\tikztotarget.west)
    -- (\tikztotarget)}
  }
}

\tikzset{
  curvedlink/.style={
    to path={
      let \p1=(\tikztostart.east), \p2=(\tikztotarget.west),
      \n1= {abs(\y2-\y1)/4} in
      (\p1) arc(90:-90:\n1) -- ([yshift=2*\n1]\p2) arc (90:270:\n1)
    },
  }
}

\usepackage{amsthm}
\usepackage{enumitem}
\usepackage{mathrsfs}
\usetikzlibrary{decorations.pathmorphing}
\usepackage[utf8]{inputenc}
\usepackage[T1]{fontenc}
\usepackage{indentfirst}
\usepackage{xcolor}
\usepackage{mathtools}
\usepackage{bm}
\usepackage{stmaryrd}
\usepackage{lineno}
\usepackage{hyperref}
\hypersetup{
    colorlinks,
    citecolor=black,
    filecolor=black,
    linkcolor=black,
    urlcolor=black
}

\usepackage{pgfplots}
\pgfplotsset{compat=1.18}
\usepackage{tikz}

\newtheorem{theorem}{Theorem}[section]
\newtheorem{lemma}[theorem]{Lemma}

\newtheorem{conjecture}[theorem]{Conjecture}
\newtheorem{proposition}[theorem]{Proposition}

\theoremstyle{definition}
\newtheorem{define}[theorem]{Definition}
\newtheorem{example}[theorem]{Example}

\newtheorem{convention}[theorem]{Convention}

\newtheorem{notation}[theorem]{Notation}

\newcommand{\GR}{K_0(\text{Var}_{\mathbb{C}})}
\newcommand{\M}{\mathcal{M}_{\mathbb{C}}}

\theoremstyle{remark}
\newtheorem{remark}[theorem]{Remark}

\numberwithin{equation}{section}

\DeclareMathOperator{\Der}{Der}

\DeclareMathOperator{\ann}{ann}

\DeclareMathOperator{\Div}{Div}

\DeclareMathOperator{\reg}{reg}

\DeclareMathOperator{\red}{red}

\DeclareMathOperator{\ord}{ord}
\DeclareMathOperator{\supp}{supp}
\DeclareMathOperator{\Sing}{Sing}
\DeclareMathOperator{\trace}{trace}
\DeclareMathOperator{\mot}{mot}
\DeclareMathOperator{\topol}{top}
\DeclareMathOperator{\Pic}{Pic}
\DeclareMathOperator{\CC}{\mathbb{C}}
\DeclareMathOperator{\PP}{\mathbb{P}}
\DeclareMathOperator{\QQ}{\mathbb{Q}}
\DeclareMathOperator{\LL}{\mathbb{L}}
\DeclareMathOperator{\Poles}{Poles}
\DeclareMathOperator{\Zeroes}{Zeroes}

\usepackage [english]{babel}
\usepackage [autostyle, english = american]{csquotes}
\MakeOuterQuote{"}

\author[D.~Bath]{Daniel Bath}
\address{\linebreak
  Daniel Bath\\
  Departement Wiskunde\\
  KU Leuven\\
  3001 Leuven, Belgium
}
\email{\href{mailto:dan.bath@kuleuven.be}{dan.bath@kuleuven.be}}
\thanks{DB was supported by FWO grant \#1282226N}

\author[W.~Veys]{Willem Veys}
\address{\linebreak
  Willem Veys\\
  Departement Wiskunde\\
  KU Leuven\\
  3001 Leuven, Belgium
}
\email{\href{mailto:wim.veys@kuleuven.be}{wim.veys@kuleuven.be}}
\thanks{WV was supported by KU Leuven Grant GYN-E4282-C16/23/010}

\title[Strong Monodromy Conjecture for certain homogeneous polynomials]{The Strong Monodromy Conjecture for a class of homogeneous polynomials in three variables}

\begin{document}
\sloppy
\begin{abstract}
We consider the class of all homogeneous, possibly non-reduced, polynomials $f$ whose associated reduced projective divisor $D_{\red} \subset \PP^{n-1}$ has (at worst) quasi-homogeneous isolated singularities. In an arbitrary number of variables $n$ and with $d$ denoting the degree of $f$, we characterize when $-n/d$ is a root of the Bernstein--Sato polynomial of $f$ in terms of elementary data involving logarithmic derivations. When we restrict to three variables, we prove the resulting class of polynomials satisfies the Strong Monodromy Conjecture, in the motivic sense.
\end{abstract}

\maketitle


\section{Introduction}

Consider a nonconstant polynomial $f \in R = \CC[x_1, \dots, x_n]$. The world of algebraic differential operators produces the Bernstein--Sato polynomial $b_f(s) \in \CC[s]$ of $f$. The Bernstein--Sato polynomial certainly has finitely many roots and these encode singularity data of $f$ relating to Milnor fibers, mixed Hodge modules, birational geometry, and more. The world of motivic integration produces the motivic zeta function $Z_f^{\mot}(s) \in \GR[\LL^{-1}][[\LL^{-s}]]$. Despite the ambient ring, Denef and Loeser \cite{DenefLoeser1998} proved that the motivic zeta function has finitely many `poles' which are all subtle singularity invariants, including, but not limited to, the negative of the log canonical threshold.

The Strong Monodromy Conjecture is wide open and promises a sharp connection between the world of differential operators and the world of motivic integration. The conjecture is: every pole of the motivic zeta function is a root of the Bernstein--Sato polynomial. There are many versions of this conjecture, including a local version  which involves a germ of $f$ and a Stronger Monodromy Conjecture, which incorporates the statement `the order of a pole must be at most its multiplicity as a root.'

In $\mathbb{A}^2$ the Strong Monodromy Conjecture, as well as this stronger variant, is true. In 1988 Loeser \cite{Loeser88} proved all but the case of order/multipicity for non-reduced $f$; in 2024 Blanco \cite{Blanco2024} completed the story.

In $\mathbb{A}^3$ or higher dimension almost nothing is known, with a few caveats. For hyperplane arrangements in $\mathbb{A}^3$ these conjectures hold (except the statement about order/multiplicity for non-reduced arrangements, which remains open), by the combined labour of \cite{BudurSaitoYuzvinsky, SaitoArrangements, uliInventiones}.
They also hold for isolated semi-quasi-homogeneous singularities in $n$ variables  \cite{BlancoBudurvdVeer}, and for a certain subclass of Newton non-degenerate polynomials in $n$ variables satisfying non-resonance conditions \cite{LoeserNewtonNonDegenerate}.

Restricting to homogeneous $f$ in three variables, the caveats become: nothing is known unless $f$ defines a hyperplane arrangement, a homogeneous isolated singularity, or a special sort of Newton non-degenerate polynomial. The first two cases are extremal instances of a class of well-studied homogeneous polynomials: those whose attached reduced projective divisor $D_{\red} \subset \PP^2$ has only quasi-homogeneous singularities. Equivalently, locally everywhere the Milnor and Tjurina numbers of $D_{\red}$ agree, see \cite{KSaitoIsolated}. Our main result is certification of the Strong Monodromy Conjecture for this class of polynomials.

\begin{restatable}{introTheorem}{thmMain} \label{thm-SMCinourCase}
    Let $f \in \CC[x_1, x_2, x_3] \setminus \CC$ be homogeneous with $D_{\red} \subset \PP^2$ the attached reduced projective curve. Suppose $D_{\red}$ has, at worst, quasi-homogeneous singularities. Then $f$ satisfies the Strong Monodromy Conjecture. Also, $f$ satisfies the local Strong Monodromy Conjecture at $0$.
\end{restatable}

\begin{convention}
    Throughout the manuscript, we never assume our polynomials or divisors are reduced, unless explicitly stated. In particular, the $f$ in Theorem \ref{thm-SMCinourCase} may be non-reduced.
\end{convention}

\bigskip

We now explain our strategy. First of all, we need a `practical' criterion for the Strong Monodromy Conjecture. While $Z_f^{\mot}(s)$ is defined using contact loci, Denef and Loeser \cite{DenefLoeser1998} gave a formula in terms of an embedded resolution of $\Div(f)$. This leads to the following criterion for certain $f$, which is probably known to specialists. (Note: the `furthermore' statement amounts to quoting previously known results.)

\begin{restatable}{introProp}{PropA} \label{prop-SMCCriterion}
    Let $f \in \mathbb{C}[x_1, \dots, x_n] \setminus \CC$ be homogeneous of degree $d$. Let $D \subset \PP^{n-1}$ be the attached divisor and $D_{\red}$ its reduced form. Suppose that $D_{\red}$ has only isolated singularities and that at each of these points $D$ satisfies the local Strong Monodromy Conjecture. Now consider the following condition:
    \begin{equation} \label{eqn-SMCCriterion}
        \bigg[ -\frac{n}{d} \in \Poles(Z_f^{\mot}(s)) \bigg] \implies \bigg[ -\frac{n}{d} \in \Zeroes(b_f(s)) \bigg].
    \end{equation}
If \eqref{eqn-SMCCriterion} holds, then $f$ satisfies the Strong Monodromy Conjecture. Similarly, if the local at $0$ version of \eqref{eqn-SMCCriterion} holds, then $f$ satisfies the local Strong Monodromy Conjecture at $0$.

Furthermore, the assumptions on $D$ are satisfied in the following cases: $n = 3$ and $D$ arbitrary; $n \geq 4$ and $D = D_{\red}$ has only quasi-homogeneous isolated singularities.
\end{restatable}

Thus our task is clear! \begin{enumerate}[label=\emph{Step} \arabic*.]
    \item (For this step we allow $n \geq 3$.) For homogeneous polynomials of degree $d$ in $n$ variables, whose attached reduced projective divisor has only quasi-homogeneous isolated singularities, find a elementary characterization of exactly when $-n/d$ is not a root of the Bernstein--Sato polynomial.
    \item Restrict to three variables and explicitly write down \emph{all} polynomials described above with $-3/d$ not a root of the Bernstein--Sato polynomial.
    \item For these special polynomials, prove that $-3/d$ is not a pole of the motivic zeta function.
\end{enumerate}

\bigskip

Let $\mathbb{A}_n$ be the Weyl algebra of $R = \mathbb{C}[x_1, \dots, x_n]$. The Bernstein--Sato polynomial $b_f(s)$ of $f$ is the $\CC[s]$-annihilator of the cyclic $\mathbb{A}_n[s]$-module
\begin{equation*}
    M_f = \frac{ \mathbb{A}_n[s]}{\ann_{\mathbb{A}_n[s]} f^s + \mathbb{A}_n[s] \cdot f}.
\end{equation*}
Under certain geometric hypotheses on $\Div(f)_{\red}$, Walther \cite{uliInventiones} (which the first author generalized in \cite{BathAnnihilationPowers} to a multi-$s$ setting) proved that $M_f$ admits a nice presentation as $\text{cokernel} (\oplus \mathbb{A}_n[s] \to \mathbb{A}_n[s])$, where the map is essentially induced by the syzygies of the partial derivatives of $f$. These syzygies can be repackaged as the module of logarithmic derivations, which is graded when $f$ is homogeneous. The non-presence of degree zero logarithmic derivations has been fruitful for finding roots of the Bernstein--Sato polynomial, see \cite{uliInventiones, BathCombinatorially, BathLQHBSpolys} and the subsequent new versions of \cite{SaitoBSProj}.

Granger and Schulze \cite{GrangerSchulze} studied the notion of semi-simple, nilpotent, and traceless logarithmic derivations (albeit for formal power series rings). As this allows for a finer accounting of degree zero derivations than recording non-existence; one can hope they yield a finer accounting of Bernstein--Sato polynomials. We show they do, obtaining a simple characterization of when $-n/d$ is a root of the Bernstein--Sato polynomial, for a class of homogeneous polynomials of interest to us.

To state this, set $b_D(s)$ the Bernstein--Sato polynomial of the divisor $D \subset \mathbb{P}^{n-1}$ attached to our homogeneous $f$, see Definition \ref{def-BSpolyProjHypersurface}. We do this for two reasons: $b_D(s)$ always divides $b_f(s)$; when $D = D_{\red}$ has quasi-homogeneous isolated singularities, a formula for the set of roots $\Zeroes(b_D(s))$ is classical. So in our case, the roots $\Zeroes(b_D(s))$ are `easy', whereas the roots $\Zeroes(b_f(s)) \setminus \Zeroes(b_D(s))$ are not.

\begin{restatable}{introTheorem}{ThmC} \label{thm-quasiHomIsolatedBSRootCharacterization}
    Let $f \in R$ be homogeneous of degree $d$, whose associated projective divisor $D \subset \PP^{n-1}$ is not a cone. Let $D_{\red}$ be the reduced divisor attached to $D$. Suppose that $D_{\red}$ has (at worst) quasi-homogeneous isolated singularities. Consider the following conditions:
    \begin{enumerate}[label=(\alph*)]
        \item $\frac{-n}{d} \notin \Zeroes(b_f(s))$;
        \item $\frac{-n}{d} \notin \Zeroes(b_D(s))$;
        \item $\Der_R(-\log_0 f)$ contains a semi-simple, non-traceless derivation.
    \end{enumerate}
    Then
    \begin{equation*}
        \bigg[\text{(a) is true} \bigg] \iff \bigg[\text{both (b) and (c) are true}\bigg].
    \end{equation*}
\end{restatable}

\bigskip

That handles Step 1. For the enumeration in Step 2, we must restrict from $\mathbb{A}^n$ to $\mathbb{A}^3$. For completely different reasons, du Plessis and Wall \cite{PlessisWall3Dim, PlessisWallOneDimSym, PlessisVersality} studied what they called $1$-symmetric projective plane curves, which is nothing more than demanding the existence of a semi-simple derivation. They enumerated these into $6$ families, but these families are not so small: for degrees $\{3,4,5,6\}$ they have some tables in \cite{PlessisWall3Dim} that grow in length at an alarming rate. We refine their enumeration, proving in Proposition \ref{prop-SymmetricReduced-StandardForm-Better} there is, up to homogeneous linear coordinate change, only one semi-simple standard form of such a polynomial, namely the one given below.


\begin{restatable}{introDef}{DefD} \label{def-standardForm}
    Let $f \in \CC[x,y,z]$ be homogeneous of degree $d$. We say $f$ is in \emph{semi-simple standard form} if $f$ factors into irreducibles as
        \begin{align} \label{eqn-explicitEnumerationEqn}
        f &= x^{a_1}y^{a_2}z^{a_3} \prod_{1 \leq q \leq m} \left( y^t + c_q x^u z^{t-u} \right)^{b_q}.
    \end{align}
    Here $m, t, u \in \mathbb{Z}_{\geq 0}$; $u < t$ and, if $u \neq 0$, then $u$ and $t$ are coprime; $a_1, a_2, a_3, b_1, \dots, b_m \in \mathbb{Z}_{\geq 0}$;  $c_1, \dots, c_m \in \CC^\times$ are pairwise distinct.
Note that $d= a_1 + a_2 + a_3 +  t(b_1 + \cdots + b_m)$.

Moreover, we say such a semi-simple standard form $f$ satisfies the \emph{non-traceless condition} when
    \begin{equation} \label{eqn-nonTracelessCondition-inIntro}
        u(d-3a_3) \neq (t-u)(d-3a_1).
    \end{equation}
\end{restatable}

Condition \eqref{eqn-nonTracelessCondition-inIntro} is enigmatic. By Proposition \ref{prop-nonTracelessCondition}, it characterizes the non normal crossing polynomials from Definition \ref{def-standardForm} admitting a non-traceless semi-simple derivation. Using it, we complete Step 2.

\begin{restatable}{introTheorem}{ThmE} \label{thm-EnumeratingEqnswithBadBSPoly}
    Suppose $f \in \CC[x, y, z]$ is homogeneous of degree $d$, such that its associated projective divisor $D \subset \PP^2$ is neither a cone nor normal crossing. Suppose that the reduced $D_{\red} \subset \PP^2$ has (at worst) quasi-homogeneous isolated singularities.

    Consider the following conditions:
    \begin{enumerate}[label=(\alph*)]
        \item $-\frac{3}{d} \notin \Zeroes(b_f(s))$;
        \item $-\frac{3}{d} \notin \Zeroes(b_D(s))$;
        \item up to a homogeneous, linear coordinate change, $f$ can be written in a semi-simple standard form (Definition \ref{def-standardForm}) satisfying \eqref{eqn-nonTracelessCondition-inIntro}, the non-traceless condition.
    \end{enumerate}
    Then
    \begin{equation*}
        \bigg[ \text{(a) is true} \bigg] \iff \bigg[\text{both (b) and (c) are true}\bigg].
    \end{equation*}
\end{restatable}

\bigskip

Step 3 sees the return of the motivic zeta function. An embedded resolution for the divisor of a homogeneous polynomial $f$ in three variables can be reconstructed  from the embedded resolution of the attached divisor in $\PP^2$. The residue at $-3/d$ for the motivic zeta function of $f$ is then encoded by a decorated normal crossing configuration on a nonsingular projective surface. The second author has proved \cite{VeysStructure} a structure theorem for such configurations which, while we do not use, inspires our computation of this residue.

We can prove the Stronger Monodromy Conjecture for a class of polynomials from Definition \ref{def-standardForm} satisfying the non-traceless condition.

\begin{restatable}{introTheorem}{ThmF} \label{thm-noBadPole}
    Let $f \in \mathbb{C}[x,y,z]$ be homogeneous of degree $d$, which is in a semi-simple standard form (Definition \ref{def-standardForm}), satisfies the non-traceless condition, and does not cut out a hyperplane arrangement. Assume that all $\frac 1{a_i}$ and $\frac 1{b_i}$ are different from $\frac 3d$.
Then $-\frac 3d$ is not a pole of $Z^{\mot}_f(s)$ nor of $Z_{f,0}^{\mot}(s)$. Moreover, $f$ satisfies the Stronger Monodromy Conjecture, as well as the local at $0$ variant.
\end{restatable}

Finally, Theorem \ref{thm-SMCinourCase} follows by the conjunction of Proposition \ref{prop-SMCCriterion}, Theorem \ref{thm-EnumeratingEqnswithBadBSPoly}, and Theorem \ref{thm-noBadPole}.

\bigskip

Here is a roadmap. In Section $2$ we define and give background about Bernstein--Sato polynomials and the motivic zeta function. As a warm-up, we prove Proposition \ref{prop-SMCCriterion}. In Section $3$, we focus on Bernstein--Sato polynomials and prove Theorem \ref{thm-quasiHomIsolatedBSRootCharacterization}. In Section $4$, we turn to the enumeration of semi-simple plane curves, Proposition \ref{prop-SymmetricReduced-StandardForm-Better}, from which Theorem \ref{thm-EnumeratingEqnswithBadBSPoly} follows. In Section $5$, we study the motivic zeta functions of the special polynomials appearing in Theorem \ref{thm-noBadPole}. This allows us to prove Theorem \ref{thm-noBadPole} and then quickly deduce Theorem \ref{thm-SMCinourCase}.


\bigskip

\emph{Acknowledgements}: We are indebted to Alexandru Dimca for kindly directing us to the work of du Plessis and Wall. And we thank Guillem Blanco for helpful conversations in the prenatal stage of this project.

\section{Bernstein--Sato polynomials and motivic zeta functions}

Let $R = \CC[x_1, \dots, x_n]$ and $f \in R \setminus \CC$. In this section we define the Bernstein--Sato polynomial of $f$ and the motivic zeta function of $f$, and we state the Strong Monodromy Conjecture. We finish with a sufficient condition for $f$ to satisfy the Strong Monodromy Conjecture, \emph{provided} that $f$ is homogeneous and the attached reduced projective divisor $D_{\red} \subset \PP^{n-1}$ has (at worst) isolated singularities.

\subsection{Bernstein--Sato Polynomials}

Let $\mathbb{A}_n$ be the Weyl algebra of $R$ and $\mathbb{A}_n[s] = \mathbb{A}_n \otimes_{\CC} \CC[s]$. For $f \in R$ introduce a formal symbol $f^s$ and denote $R[s,1/f] \otimes_{\mathbb{C}} \CC \cdot f^s$ by $R[s,1/f] f^s$. The natural left $\mathbb{A}_n[s]$-action on $R[s,1/f]$ graduates to an action on $R[s, 1/f] f^s$ by letting $R$-derivations act on $f^s$ by formal application of the chain rule. The $\mathbb{A}_n[s]$-submodule generated by $1 \otimes f^s$ (resp. $f \otimes f^s)$ is called $\mathbb{A}_n[s] \cdot f^s$ (resp. $\mathbb{A}_n[s] \cdot f^{s+1})$.

\begin{define} \label{def-BSpoly}
    Let $f \in R$. The Bernstein--Sato polynomial $b_f(s) \in \CC[s]$ is the minimal monic polynomial satisfying the \emph{functional equation}
    \begin{equation} \label{eqn-functionalEqn}
        b_f(s) f^s \in \mathbb{A}_n[s] \cdot f^{s+1}.
    \end{equation}
    The roots of $b_f(s)$ are denoted by $\Zeroes(b_f(s))$.
\end{define}

There are also local algebraic and local analytic Bernstein--Sato polynomials: repeat the construction, now using $g \in \mathscr{O}_{X,\mathfrak{x}}$ and stalks of sheaves of differential operators. The local \emph{Bernstein--Sato polynomial} of $g$ is the $\mathbb{C}[s]$-polynomial $b_{g,\mathfrak{x}}(s)$ satisfying the stalk version of the functional equation \eqref{eqn-functionalEqn}.

The Bernstein--Sato polynomial's existence is by no means obvious: in the global algebraic case it is due to Bernstein \cite{Bernstein}; in the local analytic to Bjork \cite{BjorkDimensions}. Originally it was invented to answer a question of Gelfand about the existence, and then the poles, of a meromorphic continuation of a certain distribution. It has proven to be ubiquitious throughout singularity theory: its roots (and multiplicities) appear in birational geometry through jumping numbers \cite{EinJumping} and rational singularities \cite{SaitoOnBFunction}, as well as in the theory of mixed Hodge modules and the Hodge filtration (for example \cite{MustataPopaQDivisors}, or \cite{BathDakin}, which has formulas germane to the situation of this manuscript). There is also an important connection to monodromy: if $F_{f,p}$ is the Milnor fiber of $f$ near $p \in f^{-1}(0)$, then winding around $\mathbb{C}^\times$ induces a geometric monodromy action on $F_{f,p}$, yielding an algebraic monodromy action on the cohomology of $F_{f,p}$. Malgrange \cite{Malgrange1983} and Kashiwara \cite{Kashiwara1983} proved that exponentiating the roots of the local Bernstein--Sato polynomial $b_{f,p}(s)$ recovers the eigenvalues of the algebraic monodromy.

Algebraically, the Bernstein--Sato polynomial is (the monic generator of) the $\mathbb{C}[s]$-annihilator of a certain $\mathbb{A}_n[s]$-module:
\begin{equation*}
    (b_f(s)) =\ann_{\mathbb{C}[s]} \frac{ \mathbb{A}_n[s] \cdot f^s}{ \mathbb{A}_n[s] \cdot f^{s+1}} = \ann_{\mathbb{C}[s]} M_f,
\end{equation*}
where
\begin{equation*}
    M_f := \frac{\mathbb{A}_n[s]}{\ann_{\mathbb{A}_n[s]} f^s + \mathbb{A}_n[s] \cdot f}.
\end{equation*}
For more context about Bernstein--Sato polynomials, especially from an algebraic perspective, we recommend the survey \cite{WaltherSurveyfs}.

\begin{remark} \label{rmk-BSfacts} \,
    \begin{enumerate}[label=(\alph*)]
        \item If $f = x_1^{a_1} \cdots x_r^{a_r} \in R$, then $b_f(s) = \prod_{1 \leq t \leq r} \prod_{1 \leq \ell \leq a_t}(s + \frac{\ell}{a_t})$.
        \item The Bernstein--Sato polynomial does not depend on choice of defining equation of the divisor. Moreover, if $f \in R$ and $g$ is an analytic defining equation of the germ of $f$ at $p$, then $b_{f, p}(s) = b_{g,p}(s)$.
        \item An algebraic (or analytic) function is smooth at $p$ if and only if $b_{f, p}(s) = s+1.$
        \item In both the algebraic and analytic contexts, as $b_{g,p}(s)$ is the $\mathbb{C}[s]$-annihilator of a certain module, it is not hard to prove that $b_{g,p}(s)$ equals the least common multiple of the local Bernstein--Sato polynomials of $g$ at $q \in U$, where $U$ is a sufficiently small open.
        \item In the algebraic context of $f \in R$, one can prove that $b_f(s) = \text{lcm}_{p \in \{f=0\}} b_{f,p}(s)$.
        \item From the above remarks, one can deduce that $-1$ is always a root of any sort of Bernstein--Sato polynomial.
        \item If $f \in R$ is homogeneous, then because of the equivariant $\CC^\times$-action, the above remarks also imply that $b_f(s) = b_{f,0}(s)$.
    \end{enumerate}
\end{remark}

In this paper we are mostly interested in homogeneous $f$, where we know that $b_f(s) = b_{f,0}(s) = \text{lcm}_{p \in \{f=0\}} b_{f,p}(s)$. As it is quite hard to determine which roots of $b_{f}(s)$ are `uniquely' supported at the origin, we make the following definition.

\begin{define} \label{def-BSpolyProjHypersurface}
    Let $f \in R$ be homogeneous and $D \subset \mathbb{P}^{n-1}$ the attached divisor. The \emph{Bernstein--Sato polynomial of $D$} is
    \begin{equation*}
        b_D(s) = \text{lcm}_{q \in D} b_{g_q,q}(s),
    \end{equation*}
    where $g_q$ is a local defining equation of $D$ at $q$.
\end{define}

\begin{remark} Let $f \in R$ be homogeneous with projective divisor $D \subset \PP^{n-1}$.
    \begin{enumerate}[label=(\alph*)]
        \item It follows from Remark \ref{rmk-BSfacts} that $b_D(s) = \text{lcm}_{p \in (\{f = 0\} \setminus 0)} b_{f,p}(s)$. Also $b_D(s)$ divides $b_f(s) = b_{f,0}(s)$.
        \item One usually assumes $D$ is not a cone without harm. Recall that if $D$ is a cone, then after a judicious change of coordinates $D$ corresponds to the divisor of a homogeneous $g \in \CC[x_1, \dots, x_{n-1}] \subset R$. In particular, $b_f(s) = b_D(s) = b_g(s)$. Eventually we will be concerned with $D_{\red}$ having isolated singularities, in which case $D$ being a cone implies $g_{\red} \in \mathbb{C}[x_1, \dots, x_{n-1}]$ is homogeneous with isolated singularity at $0$. This is one of the few instances where a formula for $b_{g_{\red}}(s)$ is known.
    \end{enumerate}
\end{remark}

\subsection{The motivic zeta function}


The motivic zeta function is some \lq geometric upgrading\rq\ of the $p$-adic Igusa zeta function. For motivation and history, we refer to \cite{NicaiseSurvey, VeysMonConjSurvey} and the references therein.
In order to define it, we first briefly introduce the Grothendieck ring of varieties.

\begin{define}
	The \emph{Grothendieck ring of complex varieties} $\GR$ is the quotient of the free abelian group generated by the symbols $[X]$, where $X$ runs over all complex varieties, by the relations
	$	[X] = [Y] $, if  $X \cong Y$, and
	$	[X] = [Y] + [X \setminus Y]$,  if  $ Y$ is Zariski-closed in  $X$.

	We always regard $\GR$ as a commutative ring where the multiplicative action is induced  by $[X] \cdot [Y] = [X \times Y]$. We set $\LL := [\mathbb{A}^1]$ and $\M:=\GR[\LL^{-1}]$, the ring obtained from $\GR$ by inverting $\LL$.
\end{define}

	Let $f \in \CC[x_1,\hdots,x_n] \setminus \CC$, such that   $f(0)=0$.

\begin{define}  Denote by $C_{m}(f)$ the {\em $m$-contact locus of $f$}, being the subvariety of the space of $m$-jets on $\CC^n$, that have contact order precisely $m$ with $\Div(f)$, and by
$C_{m,0}(f)$ those $m$-jets that are moreover  attached at the origin.
 The {\em motivic zeta function of $f$} is
\begin{equation}\label{zetaformula}
Z_f^{\mot}(s) := \sum_{m\geq 0} [C_{m}(f)] \LL^{-mn} (\LL^{-s})^m
\in \M[[\LL^{-s}]],
\end{equation}
 where $\LL^{-s}$ is a formal variable.
The {\em local motivic zeta function of $f$}, denoted by $Z_{f,0}^{\mot}(s)$, is given by a similar formula, replacing $[C_{m}(f)]$ by $[C_{m,0}(f)]$.
\end{define}
The notation $\LL^{-s}$ arose from the analogy with the $p$-adic setting; the $p$-adic Igusa zeta function is given by a similar series in $p^{-s}$.
We refer to e.g. \cite{DenefLoeser1998, VeysIgusaSurvey} for details. We will not really use the definition in the sequel.

Imagine for a moment a more innocent formal power series $q(s) \in \QQ[[s]]$. To say $q(s) \in \QQ(s) = \text{Frac}(\QQ[s])$ is equivalent to saying there are a finite number of poles $p_1, \dots, p_m \in \QQ$ such that $q(s) \in \QQ[s, \frac{1}{s - p_1}, \dots, \frac{1}{s-p_m}]$. Each pole $p_k$ has an order $\ell$; its the smallest $\ell$ with $(s-p_k)^\ell q(s) \in \QQ[s, \frac{1}{s - p_1}, \dots, \frac{1}{s-p_{k-1}}, \frac{1}{s-p_{k+1}}, \dots, \frac{1}{s-p_r}]$. Similar statements apply to formal power series in $A[[s]]$, provided $A$ is a integral domain and we use the field of fractions $\text{Frac}(A[s])$. The situation for $Z_f^{\mot}(s) \in \M[[\LL^{-s}]]$ is quite different: not only does $\GR$ fail to be an integral domain \cite{poonen}, but even $\LL \in \GR$ is a zero divisor \cite{borisov}. Nevertheless, Denef and Loeser proved \cite{DenefLoeser1998} that there is a finite set $S \subset \mathbb{Z}_{> 0} \times \mathbb{Z}_{> 0}$ such that
\begin{equation} \label{eqn:motivicPoles}
    Z_f^{\mot}(s) \in \M \left[ \left \{ \frac{1}{\LL^{a + bs} - 1} \right\}_{(a,b) \in S} \right] \subset \M[[\LL^{-s}]].
\end{equation}
So the motivic zeta function is `rational' in some sense, leading to the notion of a pole and the order of a pole of $Z_f^{\mot}(s)$ and $ Z_{f,0}^{\mot}(s)$.

\begin{remark}
    There are clearly subtleties concerning the notion of pole and pole order for $Z_f^{\mot}(s)$ and $ Z_{f,0}^{\mot}(s)$. We refer to \cite[Section 4]{RodriguesVeys} for an exact definition of $\LL^{-s_0}$ being a pole of the motivic zeta function, where $s_0$ is a rational number, and of its pole order. In the sequel we will shorten terminology by simply saying that \lq $s_0$ is a pole of $Z_f^{\mot}(s)$ or $ Z_{f,0}^{\mot}(s)$\rq.
\end{remark}

\smallskip
There is a formula for these zeta functions in terms of an embedded resolution  $h:Y \rightarrow \CC^n$  of $\Div(f)\subset \CC^n$. We mean by this that $h$ is a proper birational morphism,  $Y$ is smooth, $h$ induces an isomorphism over $\CC^n \setminus f^{-1}\{0\}$, and $h^{-1}(f^{-1}\{0\})=\cup_{j \in T} E_j$ is a simple normal crossing divisor, i.e., the irreducible components $E_j$ are non-singular hypersurfaces that intersect transversally.

 The \emph{numerical data} of each component $E_j$ are $(N_j,\nu_j)$, where $N_j$ and $\nu_j-1$ are the multiplicities of $E_j$ in the divisor defined by $f \circ h$ and $h^*(dx \wedge \hdots \wedge dx_n)$, respectively.
That is,
	$$\Div(f \circ h) = \sum_{j \in T} N_jE_j\quad \text{and}\quad K_Y=\Div(h^*(dx_1 \wedge \hdots \wedge dx_n)) = \sum_{j \in T} (\nu_j-1)E_j,$$
where $K_{\bullet}$ denotes the canonical divisor.
	We also set  $E^\circ_I = \cap_{i \in I} E_i \setminus \cup_{k \notin I} E_k$ for every subset $I$ of $T$.
	Note that $Y$ is the disjoint union of all $E_I^\circ$, $I \subset T$, and that $h$ induces an isomorphism between $E_{\emptyset}^{\circ}$ and $\mathbb{C}^n \setminus f^{-1}\{0\}$.

\begin{theorem} \label{thm:motivicResolutionFml} (\cite{DenefLoeser1998}) Using notation above, we have that
$$	Z_f^{\mot}(s) = \LL^{-n} \sum_{I \subset T} [E_I^\circ ]\prod_{i \in I} \frac{\LL-1}{\LL^{\nu_i+N_i s}-1}
$$
and
$$	Z_{f,0}^{\mot}(s) =  \LL^{-n} \sum_{I \subset T} [E_I^\circ \cap h^{-1}(0)]\prod_{i \in I} \frac{\LL-1}{\LL^{\nu_i+N_i s}-1}.
$$
\end{theorem}

There are other zeta functions attached to $f$. One is the topological zeta function, first introduced in \cite{DenefLoeser1992}.

\begin{define}
	The \emph{(local) topological zeta function} of $f$ is
\begin{equation}
	Z_{f,0}^{\text{top}}(s) := \sum_{I \subset T} \chi(E_I^\circ \cap h^{-1}(0)) \prod_{i \in I}\frac{1}{\nu_i+N_is}  \in \QQ(s),
	\end{equation}
where $\chi(.)$ denotes the topological Euler characteristic.
\end{define}


There is also the $p$-adic zeta function, though we will not pursue this direction. In any case, both the topological and $p$-adic zeta functions (for all but finitely many $p$) are specializations of the motivic one \cite[Section~2]{DenefLoeser1998}, and so the poles (resp. real parts of the poles) of the topological (resp. $p$-adic) zeta functions are contained in the poles of the motivic zeta function. Note that the inclusion may be strict.

For any of these zeta functions, it is very difficult to understand these poles. Nevertheless, when $n=2$, the second author determined in \cite[Theorem 4.3]{VeysDetermination} all the poles of the topological (and implicitly also motivic) zeta functions.
\begin{theorem}\label{all poles}
	Let $h:Y \rightarrow \CC^2$ be the minimal embedded resolution of the germ of $f^{-1}\{0\}$ at $0$. We have that $s_0$ is a pole of $Z_{f,0}^{\topol}(s)$ or $Z_{f,0}^{\mot}(s)$ if and only if $s_0 = -\frac{\nu_j}{N_j}$ for some component $E_j$ of the strict transform or for some exceptional component $E_j$ intersecting at least three times other components.
\end{theorem}

\subsection{The Strong Monodromy Conjecture}

The roots of the (local) Bernstein--Sato polynomial and the poles of the (local) motivic zeta function are (germ) singularity invariants. The \emph{Strong Monodromy Conjecture} suggests a tight relationship between the two sets.

\begin{conjecture}[(Global/Local) Strong Monodromy Conjecture]
	Let $f \in \CC[x_1,\hdots,x_n] \setminus \mathbb{C}$ and assume $f(0) = 0$. Then:
\begin{enumerate}[label=(\alph*)]
    \item $\Poles(Z_f^{\mot}(s)) \subset \Zeroes(b_f(s))$ (global version);
    \item $\Poles(Z_{f,0}^{\mot}(s)) \subset \Zeroes(b_{f,0}(s))$ (local version at $0$).
\end{enumerate}
\end{conjecture}

There are several variants of this conjecture, all of which have global and local versions. One can weaken the containment: the \emph{Monodromy Conjecture} posits the same inclusions hold after exponentiating both the root set and the pole set. The weaker conjecture motivates the name, as we have seen that exponentiating $\Zeroes(b_{f,0}(s))$ recovers the eigenvalues of the algebraic monodromy of the Milnor fibers near $0$. One can also strengthen the containment: the \emph{Stronger Monodromy Conjecture} says that, if $s_0$ is a pole of order $k$, then $s_0$ is a root of multiplicity at least $k$. One can also replace the motivic zeta function with the topological or $p$-adic one. Since the latter are specializations of the former, the motivic version of whatever conjecture you are interested in will imply the topological or $p$-adic one.

As discussed in the introduction, when $n = 2$ the Strong Monodromy Conjecture is true \cite{Loeser88, Blanco2024}, and otherwise it is wide open; even if we restrict to the `easier' case of homogeneous polynomials in $\mathbb{A}^3$, we only have certainty in general for hyperplane arrangements, isolated singularities, and certain special Newton non-degenerate polynomials. One advantage of considering homogeneous polynomials is they have a natural embedded resolution. As this will be used more than once, we give some details.

In what follows, let $f \in \CC[x_1, \dots, x_n] \setminus \CC$ be homogeneous with attached divisor $D \subset \PP^{n-1}$. Blow up the origin to create an exceptional divisor $F \simeq \PP^{n-1}$. Denote by $D_0$ the restriction of the strict transform (with multiplicities) of $\Div(f)$ to $F$ and note that $D_0 \simeq D$. Now take an embedded resolution $\pi: F_0 \to F$ of $D_0 \subset F$, and denote by $\{E_j\}_{j \in T}$ the irreducible components of $\pi^{-1}(D_0)$. This induces an embedded resolution $h : Y \to \CC^n$ of $\Div(f) \subset \CC^n$, for which the irreducible components (without multiplicity) of $h^{-1} \Div(f)$ are $F_0$ and $\{F_j\}_{j \in T}$, where $F_j$ is an $\mathbb{A}^1$-bundle over the corresponding $E_j$ for all $j \in T$.  For such $j \in T$, the $\mathbb{A}^1$-bundle structure implies that the numerical data of $F_j$ with respect to $h$ are the same as the numerical data of $E_j$ with respect to $\pi$.

\begin{lemma} \label{lem-GlobalZetaHomogeneousFml}
    Let $f \in \CC[x_1, \dots, x_n] \setminus \CC$ be homogeneous of degree $d$ and use preceding notation/construction. Then
    \begin{equation} \label{eqn-GlobalZetaHomogeneousFml}
        Z_f^{\mot}(s) = \LL^{-n} \frac{ (\LL - 1)\LL^{n + ds}}{\LL^{n + ds} - 1} \sum_{I \subset T} [E_I^\circ] \prod_{j \in I} \frac{ \LL - 1}{\LL^{\nu_j +  N_j s} - 1},
    \end{equation}
where for $j \in T$ the $(N_j, \nu_j)$ are the numerical data of the component $E_j$ with respect to $\pi: F_0 \to F$.

Deleting the factor $\LL^{n+ds}$ from \eqref{eqn-GlobalZetaHomogeneousFml} gives a formula for $Z_{f,0}^{\mot}(s)$.
\end{lemma}

\begin{proof}
    The following three facts are routine.
    \begin{enumerate}[label=(\roman*)]
        \item If $I \subset T$, then $[F_I^\circ] = [F_{I \cup \{0\}}^{\circ}] (\LL - 1) = [E_I^\circ](\LL - 1).$ (This includes the case $I = \emptyset$.)  Indeed, since each $F_j$ is an $\mathbb{A}^1$-bundle over its paired $E_j$, $[F_j] = [E_j \times \mathbb{A}^1] = [E_j] \LL$, and similarly for $I \subset T$. The fact easily follows.
        \item $1 + \frac{1}{\LL^{n + ds} - 1} = \frac{\LL^{n + ds}}{\LL^{n + ds} - 1}$.
        \item The numerical data of $F_0$ are $(N_0, \nu_0) = (d, n)$.
    \end{enumerate}
    Now decompose the formula of Theorem \ref{thm:motivicResolutionFml} for $Z_f^{\mot}(s)$ into two summands: one enumerating index sets $I \subset T \cup \{0\}$ containing $0$, and one enumerating the rest. Using the three facts proves the claim for $Z_f^{\mot}(s)$. The justification of $Z_{f,0}^{\mot}(s)$ is easier, since only index sets $I$ containing $0$ appear in Theorem \ref{thm:motivicResolutionFml}'s expression for $Z_{f,0}^{\mot}(s)$.
\end{proof}

In our case of interest, we need some way to control the chaos of poles.

\PropA*

\begin{proof}
We will use Lemma \ref{lem-GlobalZetaHomogeneousFml} and its notation. We first prove the global claim.

We are entitled to, and so we will, assume the embedded resolution $\pi: F_0 \to F$ of $D_0 \simeq D$ is an isomorphism outside the singular locus of $D_{0,\red} \simeq D_{\red}$. Let $\{p_\ell\}_{\ell \in L}$ be the finite set of singular points of $D_{\red}$ 
with corresponding local defining equations $\{g_{p_{\ell}}\}_{\ell \in L}$.  Write $D = \sum_{m \in M} a_m D_m$; 
its irreducible components $D_m$
are in bijection with strict transforms of $\pi$, and so $M \subset T$ enumerates these transforms. Moreover, if $m \in M$, the numerical data of $E_m$ are  $N_m = a_m$ and $\nu_m = 1$. Therefore
\begin{align*}
    &\sum_{I \subset T} [E_I^\circ] \prod_{i \in I} \frac{ \LL - 1}{\LL^{\nu_i +  N_i s} - 1} \\
    &= E_{\emptyset}^\circ + \sum_{\ell \in L} \sum_{I \subset T} [E_I^\circ \cap \pi^{-1}(p_{\ell})] \prod_{i \in I} \frac{ \LL - 1}{\LL^{\nu_i +  N_i s} - 1}
    + \sum_{\substack{M^\prime \subset M \\ M^\prime \neq \emptyset}} [E_{M^\prime}^{\circ}] \prod_{m \in M^\prime} \frac{ \LL - 1}{\LL^{\nu_m +  N_m s} - 1} \\
    &= E_{\emptyset}^{\circ} + \sum_{\ell \in L} \LL^{n-1} Z_{g_{p_{\ell}}, p_{\ell}}^{\mot}(s)
    +\sum_{\substack{M^\prime \subset M \\ M^\prime \neq \emptyset}} [E_{M^\prime}^{\circ}] \prod_{m \in M^\prime} \frac{ \LL - 1}{\LL^{1 +  a_m s} - 1}.
\end{align*}
Comparing this to \eqref{eqn-GlobalZetaHomogeneousFml} we deduce that
\begin{align} \label{eqn-motZetafmlinpieces}
    Z_f^{\mot}(s) = \LL^{-n} \frac{ (\LL - 1) (\LL^{n + ds})}{\LL^{n + ds} - 1} \bigg( [E_{\emptyset}^{\circ}] &+ \sum_{\ell \in L} \LL^{n-1} Z_{g_{p_{\ell}}, p_{\ell}}^{\mot}(s) \\
    &+\sum_{\substack{M^\prime \subset M \\ M^\prime \neq \emptyset}} [E_{M^\prime}^{\circ}] \prod_{m \in M^\prime} \frac{ \LL - 1}{\LL^{1 +  a_m s} - 1} \bigg) \nonumber
\end{align}
and that
\begin{equation*}
    \Poles(Z_f^{\mot}(s)) \subset \{-\frac{n}{d}\} \cup \bigg( \bigcup_{\ell \in L} \Poles(Z_{g_{p_{\ell}}, p_{\ell}}^{\mot}(s)) \bigg) \cup \bigg( \bigcup_{m \in M} \{ -\frac{1}{a_m} \} \bigg).
\end{equation*}
By Remark \ref{rmk-BSfacts}, the Bernstein--Sato polynomial of $f$ is the least common multiple of all its local Bernstein--Sato polynomials. These local Bernstein--Sato polynomials include $b_{g_{p_{\ell}}, p_{\ell}}(s)$, for all $\ell \in L$. They also include $b_{f_{q_m}, q_m}(s) = (s+ 1/a_m) \cdots (s + (a_m-1)/a_m)(s+1)$, where $q_m$ is a generic point of $D_m$, see Remark \ref{rmk-BSfacts}.
Since we have assumed that $\Poles(Z_{g_{p_{\ell}}, p_{\ell}}^{\mot}(s)) \subset \Zeroes(b_{g_{p_{\ell}}, p_{\ell}}(s))$, we see that
\begin{equation*}
    \Poles(Z_f^{\mot}(s)) \setminus \{- \frac{n}{d}\} \subset \Zeroes(b_f(s)).
\end{equation*}
This completes the proof about the global Strong Monodromy Conjecture.

The corresponding claim for the local Strong Monodromy Conjecture at $0$ is entirely similar. Since deleting the $\LL^{n+ds}$ factor from \eqref{eqn-GlobalZetaHomogeneousFml} gives $Z_{f,0}^{\mot}(s)$, deleting the same factor from \eqref{eqn-motZetafmlinpieces} also gives $Z_{f,0}^{\mot}(s)$. As $b_{f}(s) = b_{f,0}(s)$, the pole versus root analysis follows exactly as before.
\end{proof}

\begin{example}
    It can happen that $-n/d$ is a zero of $b_f(s)$ but not a pole of $Z_f^{\mot}(s)$, even when all the assumptions of Proposition \ref{prop-SMCCriterion} are satisfied. A non-reduced example is $f = xy(x-y)z^2(x-z)^4$, see \cite[Appendix, Ex.~A.1]{BudurSaitoYuzvinsky}.
\end{example}

\section{Concerning a root of the Bernstein--Sato polynomial}

In this section, $R = \mathbb{C}[x_1, \dots, x_n]$ with irrelevant ideal $\mathfrak{m}$. Unless otherwise stated, we endow $R$ with its standard grading, where $\deg(x_i) = 1$ for $1 \leq i \leq n$ and we use $\deg(-)$ when referring to this grading.\footnote{See subsection \ref{subsection-lqh} for a discussion of alternative positive gradings.} The degree $\leq \ell$ elements of $R$ are denote by $R_{\leq \ell}$; the homogeneous degree $\ell$ elements are denoted by $R_{\ell}$. Finally, $f \in R$ is reserved for homogeneous polynomials of degree $d$. We set $f_{\red}$ to be the reduced counterpart of $f$.

Guided by Proposition \ref{prop-SMCCriterion}, we require a criterion for $-n/d$ to be a root and to not be a root of the Bernstein--Sato polynomial $f$. When the associated reduced projective divisor has (at worst) isolated quasi-homogeneous singularities, we characterize the root-hood of $-n/d$ completely in Theorem \ref{thm-quasiHomIsolatedBSRootCharacterization}. The equivalence involves syzygy data of the Jacobian ideal of $f$; equivalently, \emph{logarithmic derivations} along $f$.

\subsection{Logarithmic derivations and a lemma}

The module of $\CC$-linear derivations on $R$ will be denoted $\Der_R$. For any $\delta \in \mathfrak{m} \cdot \Der_R$ we have $\delta \bullet \mathfrak{m}^k \subset \mathfrak{m}^{k}$. 


\begin{define} \label{def-trace,nilpotent,semi-simple}
    Let $\delta \in \mathfrak{m} \cdot \Der_R$. We call $\delta^{(2)}$ the induced $\CC$-map of vector spaces
    \begin{equation*}
        \delta^{(2)} : \frac{\mathfrak{m}}{\mathfrak{m}^{2}} \to \frac{\mathfrak{m}}{\mathfrak{m}^2}.
    \end{equation*}
    We say $\delta$ is \emph{semi-simple} (resp. \emph{nilpotent}) when $\delta^{(2)}$ is semi-simple (resp. nilpotent). The \emph{trace} $\trace(\delta)$ of $\delta$ is the trace of $\delta^{(2)}$.
\end{define}

Fixing coordinates $\{x_i\}$ and dual coordinates $\{\partial_i\}$ gives an isomorphism $\Der_R \simeq \bigoplus_i R \cdot \partial_i$. Promote the grading of $R$ to a grading of $\Der_R$: decree all $\{\partial_i\}$ to have degree $-1$. We let $\Der_{R,\leq \ell}$ (resp. $\Der_{R, \ell})$ denote the degree $\leq \ell$ (resp. homogeneous degree $\ell$) elements of $\Der_R$, and $\deg(\delta)$ the degree of a derivation. (Note: all we are doing is identifying $\Der_R \simeq \bigoplus R(-1)$.)

We may uniquely write $\Der_R \ni \delta = \sum_{-1 \leq k < \infty} \delta_k$, where $\delta_k$ is homogeneous of degree $k$. We are especially interested in the zeroeth homogeneous piece. For notational efficiency, here and throughout let $\underline{[\bm{b}]} = \begin{bmatrix}
    b_1 &\cdots & b_n
\end{bmatrix}$ denote a row matrix and $\overline{[\bm{b}]} = \begin{bmatrix}
   b_1 & \cdots & b_n
\end{bmatrix}^T$ a column matrix. Then
\begin{equation} \label{eqn-derLinearPart}
    \delta_0 = \underline{[\bm{x}]}
    Q
    \overline{[\bm{\partial_x}]}
\end{equation}
for $Q \in \text{Mat}_{n, n}(\CC)$. Moreover,
\begin{align*}
    \delta \in \mathfrak{m} \cdot \Der_R &\text{ is semi-simple (resp. nilpotent)} \\
    &\iff \delta_0 \text{ is semi-simple (resp. nilpotent)} \\
    &\iff Q \text{ is semi-simple (resp. nilpotent)}.
\end{align*}
Also: the trace of $\delta \in \mathfrak{m} \cdot \Der_R$ equals the trace of $Q$.

Since $Q$ has a Jordan--Chevalley decomposition we may define the following.


\begin{define}
    Let $\delta \in \mathfrak{m} \cdot \Der_R$ and write $\delta = \sum_{0 \leq k < \infty} \delta_k$ where $\delta_0 = \underline{[\bm{x}]}
    Q
    \overline{[\bm{\partial_x}]}$. Let $Q = Q_S + Q_N$ be the Jordan--Chevalley decomposition of $Q$ into its semi-simple and nilpotent parts. Define derivations $\delta_S = \underline{[\bm{x}]}
    Q_S
    \overline{[\bm{\partial_x}]}$ and $\delta_N = \delta - \delta_S$. Then the \emph{weak Jordan--Chevalley decomposition} of $\delta$ is $\delta = \delta_S + \delta_N$. Note: $\delta_S$ is semi-simple, homogeneous of degree $0$; $\delta_N$ is nilpotent; $[\delta_S, \delta_{N, 0}] = 0$.
\end{define}

\begin{remark} \,
    \begin{enumerate}[label=(\alph*)]
        \item We call this decomposition \emph{weak} since it is really the true Jordan--Chevalley decomposition of $\delta^{(2)}$, repackaged in terms of $\delta$. For instance, there is no reason to expect $[\delta_S, \delta_N] = 0$. However, our weak Jordan--Chevalley decomposition is unique, since the decomposition of $\delta^{(2)}$ is unique.
        \item Throughout the manuscript, any time we select $\delta \in \mathfrak{m} \cdot \Der_R$ we usually are implicitly assuming $\delta \neq 0$.
    \end{enumerate}
\end{remark}

We also will need the following special case of semi-simple derivations.

\begin{define} \label{def-diagonal}
    Fix coordinates $(x; \partial_x)$. A derivation $\delta$ is \emph{diagonal}, with respect to these coordinates, when $\delta =  \delta_0 = \underline{[\bm{x}]}
    Q
    \overline{[\bm{\partial_x}]}$, with $Q$ a diagonal matrix.
\end{define}

\begin{remark} \label{rmk-coordChangeDer} \enspace
    \begin{enumerate}[label=(\alph*)]
        \item A derivation being semi-simple or nilpotent is independent of choice of coordinates. The trace of a derivation, being the sum of its eigenvalues, is also independent of choice of coordinates. A derivation being diagonal is dependent on choice of coordinates.
        \item Let $y_1, \dots, y_n$ be homogeneous minimal generators of the irrelevant ideal $\mathfrak{m}$. There exists $A = (a_{ij})_{1 \leq i,j \leq n} \in \text{GL}_{n}(\CC)$ such that $\overline{[\bm{y}]} = A \overline{[\bm{x}]}$. Note that $A$ is the Jacobian matrix $\begin{bmatrix} \partial y_i / \partial x_j \end{bmatrix}$. So $A^T \overline{[\bm{\partial_y}]} = \overline{[\bm{\partial_x}]}$.

        Let $\delta \in \mathfrak{m} \cdot \Der_R$. Let $\delta_{(x;\partial_x)}$ and $\delta_{(x:\partial_y)}$ refer to the expression of $\delta$ in terms of the coordinates $(x;\partial_x)$ and $(y;\partial_y)$, respectively. We have the following expression for the degree zero part of $\delta$:
        \begin{align} \label{eqn-wellDefn-linChangeCoords}
            [\delta]_{(x ; \partial_x), 0}
            = \underline{[\bm{x}]} Q \overline{[\bm{\partial_x}]}
            &= \bigg( A^{-1} \overline{[\bm{y}]} \bigg)^T Q \bigg( A^T \overline{[\bm{\partial_y}]} \bigg) \\
            &= \underline{[\bm{y}]} \bigg( (A^{-1})^T Q A^T \bigg) \overline{[\bm{\partial_y}]}
            = [\delta]_{(y ; \partial_y), 0} \nonumber
\end{align}
    \item Let $E = \sum x_i \partial_i$ be the Euler derivation and $\delta \in \mathfrak{m} \cdot \Der_R$ a semi-simple derivation. The preceding item shows there exists a set of homogeneous coordinates $(y;\partial_y)$ so that $E_{(y;\partial_y)} = \sum y_i \partial_y$ and $\delta_{(y;\partial_y),0}$ are both diagonal.
    \end{enumerate}
\end{remark}

Our interest is actually in the following submodule of $\Der_R$.

\begin{define}
    Let $h \in R$. The module of \emph{logarithmic derivations} on $h$ is
    \begin{equation*}
        \Der_R(-\log h) = \{\delta \in \Der_R \mid \delta \bullet h \in (h)\}.
    \end{equation*}
    It is easy to check that $\Der_R(-\log h)$ does not depend on the choice of generator of the $R$-ideal $(h)$.

    We also define the \emph{annihilating logarithmic derivations}
    \begin{equation*}
        \Der_R(-\log_0 h) = \{\delta \in \Der_R \mid \delta \bullet h = 0\}.
    \end{equation*}
    Since $\CC^\times$ are the only units in $R$, $\Der_R(-\log_0 h)$ also does not depend on the choice of generator of the ideal $R \cdot h$.
\end{define}

\begin{remark} \label{rmk-logDerFactorization}
    Let $h \in R$ have a factorization $h = h_1^{m_1} \cdots h_r^{m_r}$ into irreducibles; let $h_{\red} = h_1 \cdots h_r$ be the reduced version of $h$. Then
    \begin{equation*}
        \Der_R(-\log h) = \cap_k \Der_R(-\log h_k^{m_k}) = \cap_k \Der_R(-\log h_k) = \Der_R(-\log h_{\red}).
    \end{equation*}
    There is no analogous identification of $\Der_R(-\log_0 h)$ and $\Der_R(-\log_0 h_{\red}).$
\end{remark}

\begin{remark} \label{rmk-LogDerBasics} Let $f \in R$ be homogeneous of degree $d$.
    \begin{enumerate}[label=(\alph*)]
        \item The Euler operator $E = \sum x_i \partial_i \in \Der_R(-\log f)$ and satisfies $E \bullet f = df$. Consequently, $\Der_R(-\log f) = R \cdot E \oplus \Der_R(-\log_0 f)$.
        \item Since $f$ is homogeneous, $\Der_R(-\log f) = \oplus_k \Der_R(-\log f)_k$.
        \item Let $D \subset \mathbb{P}^{n-1}$ be the attached projective divisor. The following are equivalent: $\Der_R(-\log f) \subset \mathfrak{m} \cdot \Der_R$; $D$ is not a cone, i.e. there is no change of coordinates and no $g \in \CC[x_1, \dots, x_{n-1}]$ such that, after said coordinate change, $R \cdot f = R \cdot g$. See \cite[Lemma 3.2]{GrangerSchulze} for details.
        \item \label{rmk-item-LogDersemi-simpleFactors} Suppose that $f = gh$ and $\delta \in \Der_R(-\log_0 gh)$ is semi-simple. Then there exists $\tau \in \Der_R(-\log_0 g)$ that is semi-simple. Indeed, change coordinates so that $\delta \in \Der_R(-\log_0 gh)_0$ is diagonal.
             Then $\frac{\delta \bullet g}{E \bullet g}E - \delta \in \Der_R(-\log_0 g)$ and its degree zero component is diagonal, hence semi-simple.
        \item \label{rmk-item-LogDerDiagonalFactors} Having fixed coordinates, the above argument shows that if $\delta \in \Der_R(-\log_0 gh)_0$ is diagonal, then there exists $\tau \in \Der_R(-\log_0 g)_0$ that is diagonal.
        \item Item \ref{rmk-item-LogDersemi-simpleFactors} (resp. item \ref{rmk-item-LogDerDiagonalFactors}) implies $\Der_R(-\log_0 f)$ contains a semi-simple (resp. diagonal) derivation if and only if $\Der_R(-\log_0 f_{\red})$ contains a semi-simple (resp. diagonal) derivation.
    \end{enumerate}
\end{remark}

Our mission is to show logarithmic derivations along a homogeneous $f$ are compatible with weak Jordan--Chevalley decompositions: namely, if $\delta \in \Der_R(-\log f) \cap \mathfrak{m} \cdot \Der_R$, then $\delta_S, \delta_N \in \Der_R(-\log f) \cap \mathfrak{m} \cdot \Der_R$. This is a consequence of \cite[Proposition 3.2]{RodriguezEulerHomogeneity}, which is the formal power series analogue.

Let $S = \CC[[x_1, \dots, x_n]]$, noting that $R \subset S$. Let $\widehat{\mathfrak{m}}$ be the maximal ideal of $S$, which is just the extension of $\mathfrak{m} \subset R$ to $S$. The $S$-derivations $\Der_S$ contain $\Der_R$. To pin down the location of various derivations, we use $\zeta \in \Der_S$,  $\delta \in \Der_R$, $\widehat{\delta}$ for the image of $\delta$ inside $S$, and $\hat{f}$ for the image of $f \in R$ inside $S$. Any $\zeta \in \Der_S$ can be uniquely written as $\zeta = \sum_{-1 \leq k < \infty} \zeta_\ell$, where $\zeta_{\ell} \in \Der_S$ is homogeneous of degree $\ell$, with $\Der_S$ weighted analogously to $\Der_R$. Finally, we can define formal logarithmic derivations as well: for $g \in S$, we define $\Der_S(-\log g) = \{\zeta \in \Der_S \mid \zeta \bullet g \in S \cdot g\}$.

There are notions of semi-simple and topologically nilpotent elements of $\Der_S$ established by \cite[Th{\'e}or{\`e}mes 1.5, 2.2 and 2.3]{GerardLevelt}. We will not define these, but we note they proved $\Der_S$ admits a Jordan--Chevalley decomposition theorem: any $\zeta \in \Der_S$ can be uniquely written as $\zeta = \zeta_S + \zeta_N$, where $\zeta_S$ is semi-simple, $\zeta_N$ is topologically nilpotent, and $[\zeta_S, \zeta_N] = 0$. We are interested in the following application of the theory.

\begin{proposition} (\cite[Proposition 3.2]{RodriguezEulerHomogeneity}) \label{prop-logDerFormalJCDecomp}
    Let $g \in S$ and $\zeta \in \Der_S(-\log g)$. Let $\zeta = \zeta_S + \zeta_N$ be the Jordan--Chevalley decomposition of $\zeta$. Then $\zeta_S, \zeta_N \in \Der_S(-\log g)$.
\end{proposition}

Deducing the same type of result for $\Der_R(-\log f)$ with $f \in R$ homogeneous is not so hard. We first need the following quick corollary.

\begin{proposition} \label{prop-weakJCdecompfromFormal}
    Let $\delta \in \mathfrak{m} \cdot \Der_R$ and $\widehat{\delta} \in \Der_S$ the image under $\Der_R \subset \Der_S$. Let $\widehat{\delta} = (\widehat{\delta})_S + (\widehat{\delta})_N$ be the Jordan--Chevalley decomposition of $\widehat{\delta}$. Then
    \begin{equation*} \label{eqn-weakJCdecompfromFormal}
        \delta = (\widehat{\delta})_{S,0} + \bigg(  (\widehat{\delta})_{N} + \sum_{1 \leq k < \infty} (\widehat{\delta})_{S,k} \bigg)
    \end{equation*}
    is the weak Jordan--Chevalley decomposition of $\delta \in \Der_R$: $(\widehat{\delta})_{S,0} = \delta_S \in \Der_R$ and $(\widehat{\delta})_{N} + \sum_{1 \leq k < \infty} (\widehat{\delta})_{S,k} = \delta_N \in \Der_R$.
\end{proposition}

\begin{proof}
    Observe that $(\widehat{\delta})_{S,0},  (\widehat{\delta})_{N} + \sum_{1 \leq k \leq \infty} (\widehat{\delta})_{S,k} \in \Der_R$, since $\delta \in \Der_R$ and $\Der_{S,0} \subset \Der_R$. By \cite[Proposition 3.1]{RodriguezEulerHomogeneity}, we know that $\delta_0 = (\widehat{\delta})_{S,0} + (\widehat{\delta})_{N,0}$ is the weak Jordan--Chevalley decomopsition of $\delta_0$. In particular,  $(\widehat{\delta})_{S,0} \in \Der_R$ is semi-simple and $(\widehat{\delta})_{N,0} \in \Der_R$ is nilpotent. Since nilpotency of elements of $\mathfrak{m} \cdot \Der_R$ is determined by their degree zero part, $(\widehat{\delta})_{N} + \sum_{1 \leq k < \infty} (\widehat{\delta})_{S,k} \in \Der_R$ is nilpotent as well.
\end{proof}

We arrive at our goal.

\begin{lemma} \label{lem-JordanChevalleyLogDer}
    Let $f \in R$ be homogeneous. Consider $\delta \in \Der_R(-\log f) \cap \mathfrak{m} \cdot \Der_R$ with its weak Jordan--Chevalley decomposition $\delta = \delta_S + \delta_N$. Then $\delta_S, \delta_N \in \Der_R(-\log f)$.

    If in addition $\delta \in \Der_R(-\log_0 f)_0$, then $\delta_S, \delta_N \in \Der_R(-\log_0 f)_0$ as well.
\end{lemma}

\begin{proof}
    By \cite[Proposition 3.2]{RodriguezEulerHomogeneity}, $(\widehat{\delta})_S, (\widehat{\delta})_N \in \Der_S(-\log \widehat{f}).$ Since $f$ is homogeneous, this implies that $(\widehat{\delta})_{S,k}, (\widehat{\delta})_{N,k} \in \Der_R(-\log f)_k$ for all $k$. Since $(\widehat{\delta})_{S,0} = \delta_S \in \Der_R$ and $(\widehat{\delta})_{N} + \sum_{\substack{-1 \leq k < \infty \\ k \neq 0}} (\widehat{\delta})_{S,k} = \delta_N \in \Der_R$ are finite sums (Proposition \ref{prop-weakJCdecompfromFormal}) of homogeneous elements of $\Der_R(-\log f)$, certainly $\delta_S, \delta_N \in \Der_R(-\log f)$.

    Now assume that $\delta \in \Der_R(-\log_0 f)_0$. We know that $\delta_N \in \Der_R(-\log f)_0$ and so $\delta_N \bullet f = \lambda f$. By \cite[Corollary 3.3]{GrangerSchulze}, we see that $\lambda = 0$. So if $\delta_N$ is nonzero, then $\delta_N \in \Der_R(-\log_0 f)_0$ and we are done.
\end{proof}

Here is how we will usually use Lemma \ref{lem-JordanChevalleyLogDer}: if $f$ is homogeneous, there is a basis $\{\delta_k\}$ of $\Der_R(-\log_0 f)$ such that each $\delta_k$ is either semi-simple or nilpotent.

\vspace{5mm}

We have defined the trace of a derivation, but have not developed the notion. What follows hints at its saliency.

\begin{remark} \label{rmk-traceLeftToRight}
    Regard $\Der_R \subset \mathbb{A}_{n}$ and write $\delta = \sum_{-1 \leq j < \infty} \delta_j$ in terms of its homogeneous components. Then, using the relations in $\mathbb{A}_{n}$, we find
    \begin{equation*}
        \delta = \sum_{-1 \leq j < \infty} \delta_j = \sum_{-1 \leq j< \infty} \sum_{1 \leq i \leq n} a_{ji} \partial_i = \sum_{-1 \leq j < \infty} \sum_{1 \leq i \leq n} (\partial_i a_{ji} - \partial_{i} \bullet a_{ji}).
    \end{equation*}
    In particular, if we write $\delta_0 = \sum_{0 \leq r, c \leq n} q_{rc} x_r \partial_c = \underline{[\bm{x}]} Q \overline{[\bm{\partial_x}]}$, we find that
    \begin{equation*}
        \delta_0 = \underline{[\bm{\partial_x}]} Q^T \overline{[\bm{x}]} - \trace(\delta_0).
    \end{equation*}
\end{remark}

\subsection{Local quasihomogeneity and symmetric projective hypersurfaces}\label{subsection-lqh}

In our applications, semi-simple non-traceless elements of $\Der_R(-\log_0 f)$ are the most pertinent. In this subsection, we try to deliver intuition about the role of semi-simple logarithmic derivations.

\begin{define} \label{def-SymHypersurfaces}
    Let $f \in R$ be homogeneous. We say
    \begin{equation*}
        \text{$f$ is $t$-symmetric when $\dim_{\mathbb{C}}\bigg(\Der_R(-\log_0 f)_0 \bigg) \geq t$}.
    \end{equation*}
    When equality holds, we will say $f$ is exactly $t$-symmetric. We refine the notion of $1$-symmetry as follows. We say
    \begin{itemize}
        \item \text{$f$ is $1$-symmetric nilpotent} when $\Der_R(-\log_0 f)_0 \ni \delta$ with $\delta \neq 0$ nilpotent;
        \item \text{$f$ is $1$-symmetric semi-simple} when $\Der_R(-\log_0 f)_0 \ni \delta$ with $\delta \neq 0$ semi-simple;
        \item \text{$f$ is $1$-symmetric non-traceless semi-simple} when $\Der_R(-\log_0 f)_0 \ni \delta \neq 0$ with $\delta$ semi-simple such that $\trace(\delta) \neq 0$.
    \end{itemize}
    Finally, if $f$ is reduced, we may instead say $D_{\red} \subset \PP^{n-1}$ has these properties.
\end{define}

\begin{remark} \label{rmk-symDerRedvNonRed} \,
    \begin{enumerate}[label=(\alph*)]
        \item The properties in Definition \ref{def-SymHypersurfaces} are properties of the ideal $(f)$, not any particular generator thereof.
        \item Let $E = \sum x_i \partial_i$ be the Euler operator. Then
    \begin{align*}
       \Der_R(-\log_0 f)_0 \oplus \CC \cdot E
       &= \Der_R(-\log f)_0 \\
       &= \Der_R(-\log f_{\red})_0 = \Der_R(-\log_0 f_{\red})_0 \oplus \CC \cdot E.
    \end{align*}
    In particular, $f$ is $1$-symmetric semi-simple if and only if $f_{\red}$ is $1$-symmetric semi-simple. And $f$ is $1$-symmetric semi-simple if and only if $\Der_R(-\log f)_0$ contains two linearly independent semi-simple derivations.
    \end{enumerate}
\end{remark}

In our applications, we care most about projective hypersurfaces with isolated singularities. This geometric assumptions constrains the amount of symmetry the hypersurface can have, as du Plessis and Wall proved.

\begin{proposition} \label{prop-degreeBound} (\cite[Lemma 2.7]{PlessisWallOneDimSym})
    Let $D_{\red} \subset \PP^{n-1}$ be reduced and $f \in R$ a defining equation of its affine cone. Then:
    \begin{align*}
        \bigg[ D_{\red} \text{ has (at worst) isolated singularities} \bigg] &\text{ and } \bigg[\deg(f) \geq 4 \bigg] \\
        &\implies \dim_{\CC} \big( \Der_R(-\log_0 f)_0 \big) \leq 1.
    \end{align*}
    In other words, if $D_{\red}$ has (at worst) isolated singularities and $\deg(D_{\red}) \geq 4$, then it is $1$-symmetric if and only if it is exactly $1$-symmetric.
\end{proposition}

Let us now connect the existence of semi-simple logarithmic derivations to other areas of math. First of all, diagonal derivations correspond to quasi-homogeneity. A vector $\bm{w} = \{w_1, \dots, w_n\} \in \mathbb{C}^n$ induces a grading on $R$: the $\bm{w}$-degree of $x_i$ is $\deg_{\bm{w}}(x_i) = w_i$. In particular, the standard grading of $R$ is the $\bm{1}$-grading. We call the $\bm{w}$ the \emph{weights} of the grading. We say the $\bm{w}$-grading is positive when $\bm{w} \in \mathbb{R}_{> 0}^{n}$.

\begin{define} \label{def:quasiHomogeneous}
    We say $g \in R$ is positively weighted homogeneous or quasi-homogeneous when there exist positive weights $\bm{w} \in \mathbb{R}_{> 0}^n$ such that $g$ is homogeneous with respect to the $\bm{w}$-grading of $R$.
\end{define}

A polynomial $g$ is $\bm{w}$-homogeneous exactly when $\delta = \sum w_i x_i \partial_i \in \Der_R(-\log g)$, in which case $\delta \bullet g = \deg_{\bm{w}}(g) g$. Conversely, suppose that $\tau = \sum a_i x_i \partial_i \in \Der_R(-\log g)$, where $a_1, \dots, a_n \in \mathbb{C}$. For every monomial $x^J$, $\tau \bullet x^J = (\sum a_i j_i)x^J$. In particular, as $J$ ranges over all monomials in the monomial support of $g$, the sum $(\sum w_i j_i)$ $\deg_{\bm{w}}(g)$ does not change. Say it equals $\lambda$. Then $g$ is $\bm{w}$-homogeneous with $\deg_{\bm{w}}(g) = \lambda$.

Just as standard homogeneity of $f \in R$ is a property of $f_{\red} \in R$, so is $\bm{w}$-homogeneity of $g \in R$.

\begin{lemma} \label{lem-quasiHomFactors}
    Factor $g = g_1^{m_1} \cdots g_r^{m_r} \in R$ into irreducibles. The following are equivalent:
    \begin{enumerate}[label=(\alph*)]
        \item $g$ is $\bm{w}$-homogeneous;
        \item each $g_k$ is $\bm{w}$-homogeneous;
        \item $g_{\red} = g_1 \cdots g_r$ is $\bm{w}$-homogeneous.
    \end{enumerate}
\end{lemma}

\begin{proof}
    Set $\delta = \sum w_i x_i \partial_i$. Then $h$ is $\bm{w}$-homogeneous if and only if $\delta \in \Der_R(-\log h)$. Now use Remark \ref{rmk-logDerFactorization}.
\end{proof}

A diagonal derivation $\delta = \sum w_i x_i \partial_i$ induces a $\CC^\star$ action on $\PP^{n-1}$:
    \begin{equation*}
        \CC^\star \times \PP^{n-1} \ni (\lambda, [p_1: \cdots : p_n]) \mapsto [\lambda^{w_1} p_1 : \cdots : \lambda^{a_n} w_n] \in \PP^{n-1}.
    \end{equation*}
For $f$ homogeneous, membership of $\delta \in \Der_R(-\log f)$ means this action fixes $D_{\red}$. This is the origin of `symmetric projective hypersurfaces' for du Plessis and Wall: they say a projective hypersurface $Z$ is $t$-symmetric when there is a $t$-parameter subgroup of $\text{PGL}_n(\CC)$ that preserves $Z$.

It is not hard to prove the following are equivalent, with $f \in R$ is homogeneous, $D \subset \PP^{n-1}$ its projective divisor, and $f_{\red}$ and $D_{\red}$ the reduced analogs:
\begin{itemize}
    \item $f$ is semi-simple $1$-symmetric;
    \item $f_{\red}$ is semi-simple $1$ symmetric;
    \item $\Der_R(-\log_0 f)$ contains a semi-simple derivation;
    \item $\Der_R(-\log_0 f_{\red})$ contains a semi-simple derivation;
    \item after some homogeneous linear coordinate change $f$ is quasi-homogeneous with respect to some weights $\bm{w}$;
    \item after some homogeneous linear coordinate change $f_{\red}$ is quasi-homogeneous with respect to some weights $\bm{w}$;
    \item after some homogeneous linear coordinate change, the $\CC^\star$ action $[p_1: \cdots : p_n] \mapsto [\lambda^{w_1} p_1 \cdots : \lambda^{w_n} p_n]$ preserves $D_{\red}$;
    \item there is a semi-simple $A \in \text{GL}_n(\CC)$ such that $\overline{A} \in \text{PGL}_n(\CC)$ fixes $D_{\red}$.
\end{itemize}

\begin{remark}
    The presence of a semi-simple non-traceless logarithmic derivation is more subtle: many of the equivalences above break down.
\end{remark}

\begin{example} \label{ex:1SymmetricNilpotent}
    Du Plessis and Wall \cite[pg.~123]{PlessisWall3Dim} proved the polynomials $f = x^e\prod_{i}(y^2 - 2xz + 4a_i x^2)$, with $a_i \in \CC^\times$ distinct and $e \in \{0,1\}$, enumerate (up to coordinate change) all $1$-symmetric nilpotent polynomials in three variables. They also proved the singularities of the attached projective divisor $D \subset \PP^2$ are quasi-homogeneous only when $\deg(f) \leq 4$. And when $\deg(f) \geq 4$, these $f$ are not semi-simple.
\end{example}

Again, in applications we often consider projective hypersurfaces with (at worst) quasi-homogeneous isolated singularities. So at every singular point, the hypersurface admits a local defining equation that is positively weighted homogeneous. These homogeneity phenomena are well-studied and later play a role for us.

\begin{define} \label{def-localQuasiHom}
    Let $G \subset X$ be an analytic or algebraic divisor on the $n$-dimensional Stein manifold $X$. We say $G$ is positively weighted homogeneous at $\mathfrak{x} \in G$ if there exists local analytic coordinates $(x, \partial_x)$ around $\mathfrak{x} \in X$ and a local defining equation $g \in R = \CC[x_1, \dots, x_n]$ of $G$ at $\mathfrak{x}$, such that $g$ is positively weighted homogeneous. We say $G$ is \emph{locally quasi-homogeneous} when it is positively weighted homogeneous at all $\mathfrak{x} \in G$. Finally, we say a polynomial or analytic function is \emph{locally quasi-homogeneous} when the associated divisor is so.
\end{define}

\begin{lemma} \label{lem-DivisorNonReduced-QuasiHom}
    Let $g \in R$ and $g_{\red}$ its reduced form; let $G \subset \CC^{n}$ be the divisor attached to $g$ and $G_{\red} \subset \CC^{n}$ the divisor attached to $f_{\red}$. Then the following are equivalent:
    \begin{enumerate}[label=(\alph*)]
        \item $g$ is locally quasihomogeneous;
        \item $g_{\red}$ is locally quasihomogeneous;
        \item $D$ is locally quasihomogeneous;
        \item $D_{\red}$ is locally quasihomogeneous.
    \end{enumerate}
\end{lemma}

\begin{proof}
    Use Proposition \ref{lem-quasiHomFactors} point-by-point.
\end{proof}

\begin{remark} \label{rmk-basicsPosWeightedHom}
    Let $D \subset X$ be an analytic or algebraic divisor.
    \begin{enumerate}[label=(\alph*)]
        \item The divisor $D$ is positively weighted homogeneous at any point in $D_{\red, \reg}$.
        \item Suppose $Z$ is reduced and has an isolated singularity at $\mathfrak{x}$. Then $Z$ is positively weighted homogeneous at $\mathfrak{x}$ if and only if the Milnor and Tjurina numbers of $Z$ at $\mathfrak{x}$ agree.
        \item Suppose that $Z \subset \PP^{n-1}$ is reduced with quasi-homogeneous isolated singularities. Then the affine cone $C(Z) \subset \CC^{n}$ is locally quasi-homogeneous. Certainly $C(Z)$ is positively weighted homogeneous at $0$: consider its standard homogeneous defining equation. For the other points, consider the standard de-homogenized charts. On such a chart the hypersurface is isomorphic to $\CC \times (D_{\red} \cap \{x_i = 1\})$. Since $D_{\red} \cap \{x_i = 1\}$ has only isolated singularities, all of which are positively weighted homogeneous, we see $C(D_{\red})$ is positively weighted homogeneous away from the origin as well.
        \item \label{item-basicsPosWeightedHom-ourCase} Consider $D \subset \mathbb{P}^{n-1}$ such that $D_{\red}$ has (at worst) isolated quasi-homogeneous singularities. Then the affine cone $C(D) \subset \mathbb{C}^n$ is locally quasi-homogeneous. This follows by the previous item and Lemma \ref{lem-DivisorNonReduced-QuasiHom}.
    \end{enumerate}
\end{remark}

\subsection{Characterizing Root-hood of $-n/d$}

We return to Bernstein--Sato polynomials. For an arbitrary homogeneous $f \in R$ we can use the language of semi-simple non-traceless logarithmic derivations to constrain the roots of $b_f(s)$ that are not roots of $b_D(s)$. In particular, we obtain a criterion for $-n/d \notin \Zeroes(b_f(s)) \setminus \Zeroes(b_D(s))$.

Recall that $M_f = \frac{\mathbb{A}_n[s]}{\ann_{\mathbb{C}[s]} f^s + \mathbb{A}_n[s] \cdot f}$ and that the Bernstein--Sato polynomial $b_f(s)$ of $f$ is (the monic generator of) the $\mathbb{C}[s]$-annihilator of $M_f$.

\begin{proposition} \label{prop-BSMultipleBeta}
    Let $f \in R$ be homogeneous of degree $d$ with associated projective divisor $D \subset \PP^{n-1}$. For each $\delta \in \Der_R(-\log_0 f)_0$, consider the following subspace of $R_p$:
    \begin{equation*}
        \Xi_{p, \delta} = \{ \sum \CC \cdot Q \mid Q \in R_p, \delta \in \Der_R(-\log Q), \delta \bullet Q \neq -\trace(\delta) Q \}.
    \end{equation*}
    In particular,
    \begin{equation*}
        \Xi_{0, \delta} = \begin{cases}
            \CC \cdot 1 \text{ if } 0 \neq \trace(\delta) \\
            0 \text{ if } 0 = \trace(\delta).
        \end{cases}
    \end{equation*}
    Also consider the span of all these subspaces of $R_p$:
    \begin{equation*}
        \Xi_p = \sum_{\delta \in \Der_R(-\log_0 f)_0} \Xi_{p, \delta},
    \end{equation*}
    where we convene that $\Xi_p = 0$ when $\Der_R(-\log_0 f)_0 = 0$.

    Define
    \begin{equation*}
        \beta(s) = \prod_{\substack{0 \leq p < (d-1)n \\ \Xi_p \neq R_p}} (s + \frac{p + n}{d}).
        \end{equation*}
    Then we have the following divisibility statement involving Bernstein--Sato polynomials:
    \begin{equation*}
        b_D(s) \mid b_f(s) \mid b_D(s) \cdot \beta(s).
    \end{equation*}
\end{proposition}

\begin{remark}
    Suppose that $\Der_R(-\log_0 f)_0$ contains a semi-simple non-traceless derivation $\tau$. Then $\Xi_{0, \tau} = \mathbb{C} \cdot 1$ and $-n/d \notin \Zeroes(\beta(s))$. In this case, Proposition \ref{prop-BSMultipleBeta} implies the equivalence $-3/d \in \Zeroes(b_f(s)) \iff -3/d \in \Zeroes(b_D(s))$.
\end{remark}


\begin{proof}
    Consider the $\mathbb{A}_n[s]$-submodule $N_f = b_D(s) \cdot M_f \subset M_f$. Recall we defined $M_f$ as a quotient $\mathbb{A}_n[s] \twoheadrightarrow M_f$. For $P(s) \in \mathbb{A}_n[s]$, let $\overline{P(s)} \in M_f$ denote its image under this map. So $M_f$ is cyclically generated by $\overline{1}$ and $N_f$ by $\overline{b_D(s)}$. Let $\mathfrak{m} \subset R$ be the irrelevant ideal; we label the left ideal extension of $\mathfrak{m}^q$ to $\mathbb{A}_n[s]$ by $(\mathfrak{m}^q)^e$. By construction, $N_f$ is supported at $0$ (see Remark \ref{rmk-BSfacts}), so for any $P(s) \in \mathbb{A}_{n}[s]$ there exists $N_{P(s)} \in \mathbb{Z}_{\geq 0}$ such that $\mathfrak{m}^{N_{P(s)}} \cdot \overline{P(s) b_D(s)} = 0$. In particular there exists $N \in \mathbb{Z}_{\geq 0}$ such that $(\mathfrak{m}^N)^e \cdot \overline{b_D(s)} = 0$ in $N_f$.

    For $0 \leq q \leq N-1$, define
    \begin{equation*}
        \beta_q(s) = \prod_{\substack{q \leq p \leq N-1 \\ \Xi_p \neq R_p}} (s + \frac{p + n}{d}).
    \end{equation*}
    We claim that
    \begin{equation} \label{eqn-claimedBetaFact}
        \frac{\beta_{q}(s)}{\beta_{q+1}(s)} \cdot (\mathfrak{m}^q)^e \cdot \overline{b_D(s)} \subset (\mathfrak{m}^{q+1})^e \cdot \overline{b_D(s)}.
    \end{equation}
    Certifying \eqref{eqn-claimedBetaFact} terminates the proof. Indeed, $\beta_0(s) \cdot \overline{b_D(s)} = 0$ by an inductive application of \eqref{eqn-claimedBetaFact}. Since $N_f$ is cyclically generated by $\overline{b_D(s)}$ and $\CC[s] \subset \mathbb{A}_n[s]$ is central, $\beta_0(s)$ annihilates $N_f$. So $b_D(s) \beta_0(s) $ kills $M_f$, which is to say $b_f(s)$ divides $b_D(s) \beta_0(s)$. As the zeroes of $b_f(s)$ lie in $(-n,0)$ \cite{KashiwaraBFunctions, SaitoMicrolocalBFunction}, we conclude that $b_f(s)$ divides $b_D(s) \beta(s)$.

    So we must prove \eqref{eqn-claimedBetaFact}. First, fix $\delta = \sum a_i \partial_i \in \Der_R(-\log_0 f)_0$ (so $a_i \in R_1$) and $Q \in R_q$. Suppose that $\delta \in \Der_R(-\log Q)$. Because of the $\bm{1}$-grading of all actors, there exists $\lambda \in \CC$ such that $\delta \bullet Q = \lambda Q$.


    Observe:
    \begin{align} \label{eqn-beta-1}
        (\mathfrak{m}^{q+1})^e \cdot \overline{b_D(s)}
        \ni \bigg((\sum_i \partial_i a_i)Q \bigg) \cdot \overline{b_D(s)}
        &= \bigg( b_D(s) \sum_i \partial_i a_i \bigg) \cdot \overline{Q} \\
        &= b_D(s)\bigg(\sum_i a_i \partial_i + \trace(\delta)\bigg) \cdot \overline{Q} \nonumber \\
        &= b_D(s) \cdot \overline{ (\delta + \trace(\delta) )Q} \nonumber \\
        &= b_D(s) \cdot \overline{Q (\delta + \lambda + \trace(\delta))} \nonumber \\
        &= \bigg(b_D(s) Q (\delta + \lambda + \trace(\delta)) \bigg) \cdot \overline{1} \nonumber \\
        &= \bigg(b_D(s) Q (\lambda + \trace(\delta)) \bigg) \cdot \overline{1} \nonumber \\
        &= \bigg((\lambda + \trace(\delta))Q\bigg) \cdot \overline{b_D(s)}. \nonumber
    \end{align}
    Note that $\delta \cdot f^s = 0$ in $M_f$, which is to say that $\delta \cdot \overline{1} = 0$ in $M_f$. So the antepenultimate equality is true.

    An entirely similar argument, now with $E = \sum_i x_i \partial_i$ and $Q \in R_q$ arbitrary, yields:
    \begin{align} \label{eqn-beta-2}
        (\mathfrak{m}^{q+1})^e \cdot \overline{b_D(s)}
        \ni \bigg( (\sum_i \partial_i x_i) Q \bigg) \cdot \overline{b_D(s)}
        &= \bigg( b_D(s) \sum \partial_i x_i \bigg) \cdot \overline{Q} \\
        &= \bigg(b_D(s)( E+ \trace(E)) \bigg)\cdot \overline{Q} \nonumber \\
        &= \bigg(b_D(s) (E + n) Q \bigg) \cdot \overline{1} \nonumber \\
        &= \bigg( b_D(s) Q (E + q + n) \bigg) \cdot \overline{1} \nonumber \\
        &= \bigg( b_D(s) Q (ds + q + n) \bigg) \cdot \overline{1} \nonumber \\
        &= \bigg((ds + q + n) Q \bigg) \cdot \overline{b_D(s)}. \nonumber
    \end{align}
    Here the penultimate equality uses that $(E - ds) \cdot f^s = 0$ in $M_f$, from whence $(E - ds) \cdot \overline{1} = 0$.

   On one hand \eqref{eqn-beta-2} implies (since in \eqref{eqn-beta-2} the choice of $Q \in R_q$ is arbitrary)
    \begin{equation*} \label{eqn-beta-3}
        (ds + q + n) (\mathfrak{m}^q)^e \cdot \overline{b_D(s)} \subset (\mathfrak{m}^{q+1})^e \cdot \overline{b_D(s)}.
    \end{equation*}
    So \eqref{eqn-claimedBetaFact} is true provided $\Xi_q \neq R_q$. It remains to prove that
    \begin{equation} \label{eqn-claimedbeta-3}
        \bigg[ \Xi_q = R_q \bigg] \implies \bigg[ (\mathfrak{m}^q)^e \cdot \overline{b_D(s)} \subset (\mathfrak{m}^{q+1})^e \cdot \overline{b_D(s)} \bigg].
    \end{equation}
    So suppose that $\Xi_q = R_q$. We may find a basis $\{Q_u\}$ of $R_q$ with the following property: for each $Q_u$ there exists $\delta_{u} \in \Der_R(-\log_0 f)_0$ and $\lambda_{u} \in \CC$, not equal to $-\trace(\delta_u)$, such that $\delta_u \bullet Q_u = \lambda_u Q_u$. By \eqref{eqn-beta-1}, every such basis element satisfies $Q_u \cdot \overline{b_D(s)} \in (\mathfrak{m}^{q+1})^e \cdot \overline{b_D(s)}$. Thus $R_p \cdot \overline{b_D(s)} = \{ \sum \CC \cdot Q_u \}\cdot \overline{b_D(s)} \subset (\mathfrak{m}^{q+1})^e \cdot \overline{b_D(s)}$ and \eqref{eqn-claimedbeta-3} follows.
\end{proof}

We want a reverse implication: a criterion on $f$ that forces $-n/d$ to be a root of $b_f(s)$. For this, we need an extra hypothesis.

Recall that the Weyl algebra comes equipped with a canonical exhaustive, increasing filtration, called the order filtration $P_\bullet^{\ord} \mathbb{A}_n$. The \emph{total order filtration} $F_\bullet^\sharp \mathbb{A}_n[s]$ on $\mathbb{A}_n[s]$ is defined as
\begin{equation*}
    F_k^{\sharp} \mathbb{A}_n[s] = \sum_{0 \leq t \leq k} P_t^{\ord} \mathbb{A}_n \otimes_{\mathbb{C}} \CC[s]_{\leq {k-t}},
\end{equation*}
where $\CC[s]_{\leq k-t}$ are the elements of degree at most $k-t$. This filtration is also exhaustive and increasing.

\begin{define} \label{def-diffLinear}
    We say $g \in R$ is of \emph{differential linear type} when
    \begin{equation*}
        \ann_{\mathbb{A}_n[s]} g^s = \mathbb{A}_n[s] \cdot \bigg( \ann_{\mathbb{A}_n[s]} g^s \cap F_1^\sharp \mathbb{A}_n[s]\bigg).
    \end{equation*}
\end{define}

\begin{remark} One can check that
\begin{equation*}
    \ann_{\mathbb{A}_n[s]} g^s \cap F_1^\sharp \mathbb{A}_n[s] = \{ \delta - s \frac{\delta \bullet g}{g} \mid \delta \in \Der_R(-\log g) \}.
\end{equation*}
\end{remark}

The differential linear type condition was first considered in \cite{CalderonMorenoNarvaezMacarroCompositio}. See \cite{NarvaezMacarroSurvey} for an (older) survey of adjacent topics, \cite[Theorem 3.26]{uliInventiones} for the largest class of known divisors with the hypothesis, and \cite[Corollary 2.30]{BathAnnihilationPowers} for a multi-$s$ generalization thereof. Let us demonstrate the relevance of the condition to the root-hood of $-n/d$. This following argument improves \cite[Theorem 5.13]{uliInventiones}, though the technique is similar. 

\begin{proposition} \label{prop-DiffLinTypeRootCriterion}
    Let $f \in R$ be homogeneous of degree $d$ whose associated projective divisor $D \subset \PP^{n-1}$ is not a cone. Suppose that $f$ is of differential linear type and that $\Der_R(-\log_0 f)$ contains no semi-simple non-traceless logarithmic derivations. Then
    \begin{equation*}
       - \frac{n}{d} \in \Zeroes(b_f(s)).
    \end{equation*}
\end{proposition}

\begin{proof}
     We claim that our hypotheses guarantee a surjection of left $\mathbb{A}_n[s]$-modules
     \begin{equation*}
         M_f = \frac{\mathbb{A}_n[s]}{\ann_{\mathbb{A}_n[s]} f^s + \mathbb{A}_n[s] \cdot f} \twoheadrightarrow L :=\frac{\mathbb{A}_n[s]}{\mathbb{A}_n[s] \cdot \mathfrak{m} + \mathbb{A}_n[s] \cdot (n + ds)}.
     \end{equation*}
     Before demonstrating this surjection, first note that $\ann_{\CC[s]} L = \CC[s] \cdot (n+ds)$ and $L \neq 0$. Indeed, inspection shows $\CC[s] \cdot (n+ds) \subset \ann_{\CC[s]} L$. Equality fails if and only if $\ann_{\CC[s]} L = \CC[s]$, in other words, if and only if $L = 0$. But, as rings, $L \simeq \mathbb{A}_n / \mathbb{A}_n \cdot \mathfrak{m} \neq 0$.

     Now we demonstrate the surjection. Since $f$ is of differential linear type, $\ann_{\mathbb{A}_n[s]} f^s = \mathbb{A}_n[s] \cdot (E - ds) + \sum \mathbb{A}_n[s] \cdot \delta$, where $E = \sum_i x_i \partial_i$ is the Euler derivation, and the second summand runs over a homogeneous generating set $\{\delta\}$ of $\Der_R(-\log_0 f)$.

     Consider $\delta = \sum_i a_i \partial_i \neq 0 $, belonging to such a generating set. Then $\delta \notin \Der_R(-\log_0 f)_{-1}$. Because if this happened, a judicious change of coordinates makes $\partial_i \bullet f = 0$, realizing $D$ as a cone. Now suppose that $\delta \in \Der_R(-\log_0 f)_0$. Using Remark \ref{rmk-traceLeftToRight} to place all the $\partial_i$'s on the left, we find, inside $\mathbb{A}_n[s]$, the equality $\delta = (\sum_i \partial_i a_i) - \trace(\delta)$. Similarly, $E = (\sum_i \partial_i x_i) - n$. Now suppose that $\delta \in \Der_R(-\log_0 f)_{\geq 1}$, so that each nonzero $a_i$ is homogeneous of degree at least two. Working in $\mathbb{A}_n[s]$ we have the equality $\delta = \sum_i (\partial_i a_i - (\partial_i \bullet a_i))$, and so $\delta \in \mathbb{A}_n[s] \cdot \mathfrak{m}$. Since $f$ is of differential linear type, putting these computations together shows that the following are equivalent:
     \begin{align*}
        \bigg[ \ann_{\mathbb{A}_n[s]} f^s &\subset \mathbb{A}_n[s] \cdot \mathfrak{m} + \mathbb{A}_n[s] \cdot (n + ds) \bigg] \\
        &\iff \bigg[ \Der_R(-\log_0 f)_0 = \bigoplus_k \CC \cdot \delta_k \text{ where each $\delta_k$ is traceless} \bigg] \\
        &\iff \bigg[ \Der_R(-\log_0 f)  \text{ contains no semi-simple non-traceless derivations} \bigg].
     \end{align*}
     (Note: we have already seen that $L \neq 0$, whence $\mathbb{A}_n[s] \cdot \mathfrak{m} + \mathbb{A}_n[s] \cdot (n + ds) \neq \mathbb{A}_n[s]$.) Since we have presumed $\Der_R(-\log_0 f)$ contains no semi-simple non-traceless logarithmic derivations, we do indeed have our surjection $M_f \twoheadrightarrow L_f$.

    This surjection basically completes the proof. Since the $\mathbb{C}[s]$-annihilator of $M_f$ sits inside the $\mathbb{C}[s]$-annihilator of any subquotients of $M_f$, we see that
     \begin{equation*}
         \mathbb{C}[s] \cdot b_f(s) = \ann_{\mathbb{C}[s]} M_f \subseteq \ann_{\mathbb{C}[s]} L = \mathbb{C}[s] \cdot (n + ds).
     \end{equation*}
\end{proof}

We have arrived at the section's main prize: a characterization of when $-n/d$ is a root of the Bernstein--Sato polynomial of $f$, provided the associated reduced projective divisor has, at worst, quasi-homogeneous isolated singularities. The point is that such $f$ are of differential linear type.

\ThmC*

\begin{remark}
    When $\deg(f_{\red}) = \deg(D_{\red}) \geq 4$, Proposition \ref{prop-degreeBound} and Remark \ref{rmk-logDerFactorization} imply that (b) is equivalent to $\Der_R(-\log_0 f)_0 = \CC \cdot \tau$, for $\tau$ a semi-simple and non-traceless logarithmic derivation.
\end{remark}


\begin{proof}
     That $(a) \implies (b)$ is standard, see Remark \ref{rmk-BSfacts}. Condition $(c)$ guarantees the existence of a semi-simple non-traceless derivation $\tau$. Recalling Proposition \ref{prop-BSMultipleBeta} and its notation, (c) implies that $\Xi_{0, \tau} = 0$ and so $-n/d \notin \Zeroes(\beta(s))$. Hence $(b)$ and $(c)$ together imply $(a)$. To complete the proof, we must confirm that $(a) \implies (c)$. If $f$ is of differential linear type, this is just Proposition \ref{prop-DiffLinTypeRootCriterion}.

    We will prove that $f$ is of differential linear type. Criteria appear in \cite[Theorem 3.26]{uliInventiones}. So (in the language of loc. cit.) it suffices to check that the affine cone $C(D) \subset \CC^n$ is Saito-holonomic, tame, and locally quasi-homogeneous. (Locally quasi-homogeneous is stronger than the condition `strongly Euler-homogeneous' appearing in loc. cit.) By Remark \ref{rmk-basicsPosWeightedHom}.\ref{item-basicsPosWeightedHom-ourCase}, $C(D)$ is locally quasi-homogeneous. Saito-holonomicity and tameness only involve the reduced structure $C(D)_{\red}$. So it suffices to prove that if $Z \subset \mathbb{P}^{n-1}$ is reduced and has (at worst) isolated singularities, then $C(Z) \subset \CC^n$ is Saito-holonomic and tame. Saito-holonomicity means the logarithmic stratification is locally finite. And $C(Z)$ has a locally finite logarithmic stratification. Indeed, its logarithmic strata are: $\{0\}$; $\{\cup_{p \in Z \setminus Z_{\reg}} C(p) \setminus 0 \}$ where $C(p) \subset \CC^n$ is the cone of $p \in \mathbb{P}^{n-1}$; the regular locus $C(Z)_{\reg}$.\footnote{Actually, any locally quasi-homogeneous divisor is automatically Saito-holonomic, see \cite[Lemma 1.6]{BathSaitoTLCT}.}
    That $C(Z)$ is tame is \cite[Corollary 1.4]{BathSaitoTLCT}.
\end{proof}

\begin{remark} \label{rmk-computingBSofIsolatedProjective}
    If $g \in R$ is reduced, positively weighted homogeneous, and has an isolated singularity, then there is well known formula for $b_g(s)$ entirely in terms of the degree sequence of the Milnor algebra $R / (g, \partial_1 \bullet g, \dots, \partial_n \bullet g)$. So for reduced $f$ and $D$ as in Theorem \ref{thm-quasiHomIsolatedBSRootCharacterization}, computation of $\Zeroes(b_D(s))$ is `routine.'
\end{remark}

\section{Semi-simple symmetric projective plane curves and their Bernstein--Sato polynomials.}

Throughout, $f \in R = \CC[x_1, \dots, x_n]$ is homogeneous of degree $d$ and $D \subset \PP^{n-1}$ its attached projective divisor. We have seen whether or not $-n/d$ is a root of the Bernstein--Sato polynomial of $f$ is related to the existence of semi-simple traceless derivations in $\Der_R(-\log_0 f)_0$. In particular, when $D_{\red}$ has (at worst) quasi-homogeneous isolated singularities, such derivations along with $b_D(s)$ characterize when $-n/d \in \Zeroes(b_{f(s)})$.

Our goal here is to give \emph{explicit} defining equations for all $f \in \mathbb{C}[x,y,z]$ such that: $C(D)$ is not a cone; $D_{\red} \subset \PP^{2}$ has (at worst) quasi-homogeneous (necessarily isolated) singularities; $-3/d$ is not a root of $b_f(s)$.

Our enumeration builds on, and then significantly improves, work of du Plessis and Wall \cite{PlessisWall3Dim}. They find a six families classification, but they do not find a way to `finitely' enumerate all family members. We enumerate all semi-simple symmetric plane curves in terms of (up to homogeneous coordinate change) a \emph{single} defining equation. Then we identify which family members do not have $-3/d$ as a root of $b_f(s)$ in Theorem \ref{thm-EnumeratingEqnswithBadBSPoly}.

\subsection{Enumeration of semi-simple $1$-symmetric projective plane curves}

Within this subsection (and only here) we reserve $f_d \in \CC[x,y,z]$ for a reduced homogeneous polynomial of degree $\deg(f_d) = d$.

Start by denoting the monomial support of $f_d$ by $\supp(f_d) \subset \mathbb{Z}_{\geq 0}^3$. Homogeneity places $\supp(f_d)$ inside the hyperplane $x + y + z = d$, making every monomial in $\supp(f_d)$ of the form $x^{t_1}y^{t_2}z^{d-(t_1 + t_2)}.$ Define
\begin{equation*}
    \phi_d : \mathbb{Z}_{\geq 0}^3 \to \mathbb{Z}^2 \quad \text{where} \quad \phi_d (a, b, c) = (-b, d - a).
\end{equation*}
We use coordinates $(u, v)$ for $\mathbb{Z}^2$. Consider the triangle
\begin{equation*}
    T_d = \{(u,v) \in \mathbb{Z}^2 \mid -d \leq u \leq 0, \, 0 \leq v \leq d, \, 0 \leq u + v\}.
\end{equation*}
It lies within the second quadrant, with hypotenuse connecting $(-d,d)$ and $(0,0)$. And $\phi_d$ maps $\supp(f_d)$ bijectively into $T_d$:
\begin{equation*}
    \phi_d(t_1, t_2, d-t_1-t_2) = (-t_2, d-t_1) \in T_d.
\end{equation*}
As subsequent arguments involve manipulating $T_d$, we have drawn $T_d$, along with some extra data, in Figure \ref{fig:lattice}. 

\begin{figure}[h]
\includegraphics[scale=0.35]{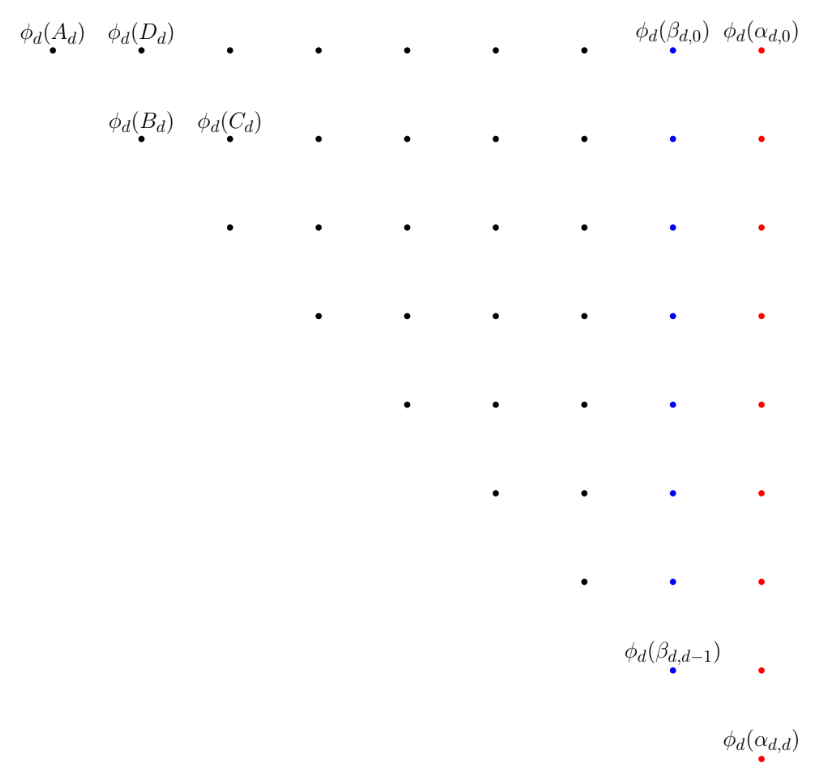}
\caption{The integral points of the triangular region $T_d$. The upper left corner is $(-d,d)$, the upper right $(0,d)$, the bottom right the origin. Many yet-to-be named integral points are drawn. The points $\phi_d(A_d), \phi_d(B_d), \phi_d(C_d),$ and $\phi_d(D_d)$ are labeled in the top left. The points $\{\phi_d(\alpha_{d,r})\}_r$ are the integral points on the line $u = 0$ and appear red. The points $\{\phi_d(\beta_{d,r})\}_r$ are the integral points on the line $u = -1$ and appear blue.}
\label{fig:lattice}
\end{figure}

\begin{proposition} \label{prop-diagDerLineSement}
    In the language of the above the following are equivalent:
    \begin{enumerate}[label=(\alph*)]
        \item there exists a diagonal $\delta \in \Der_R(-\log_0 f_d)_0$;
        \item $\phi_d (\supp(f_d))$ is contained in a line.
    \end{enumerate}
\end{proposition}

\begin{proof}
    Write the diagonal $\delta$ as $w_1 x \partial_x + w_2 y \partial_y + w_3 z \partial_z$. Then any monomial $x^{t_1}y^{t_2} z^{d - (t_1 + t_2)}$ in $\supp(f_d)$ satisfies $w_1 t_1 + w_2 t_2 + w_3 (d-t_1 - t_2) = 0$, which is to say that $\phi_d(t_1, t_2, d - (t_1 - t_2)) = (-t_2, d - t_1)$ lies on the line $w_1 v + w_2 u = w_1 d + w_3 (v + u)$.

    Conversely, if $\phi_d(\supp(f_d))$ lies in the line $v = mu + b$, then every monomial $x^{t_1} y^{t_2} z^{d-t_1-t_2}$ in $\supp(f)$ satisfies $d - t_1 = -m t_2 + b$. So the diagonal derivation $(x \partial_x - m y \partial_y) - (d-b)E$ belongs to $\Der_R(-\log_0 f_d)_0$. Similarly, if $\phi_d(\supp(f_d))$ lies in the line $u = 0$,  then $\supp(f_d) \subseteq \{x^{t_1}y^0z^{d-t_1}\}_{0 \leq t_1 \leq d}$ and so $(d-t_1)x \partial_x - t_1 z \partial_z$ kills $f_d$.
\end{proof}

If $f_d$ defines a $1$-symmetric semi-simple $Z \subset \PP^2$, then after a homogeneous coordinate change we may assume that $f_d$ admits a
nonzero diagonal derivation $\delta \in \Der_R(-\log_0 f_d)$. Fix these coordinates. By the previous proposition, $\phi_d(\supp(f_d))$ lies on a line in $T_d$. Let $L$ denote the minimal line segment containing $\phi_d(\supp (f_d))$. We assume $L$ is not a point, since when it is $f_d$ is a monomial. 

For each edge of the triangle $T_d$, there exists an endpoint of $L$ of distance $\leq 1$ to that edge. For example, consider the hypotenuse $\{u + v = 0\}$. A monomial $x^a y^b z^c$ lands in the region $u + v \geq 2$ exactly when $c \geq 2$. So
\begin{equation*}
    \bigg[ \phi_d(\supp(f_d)) \subset \bigg( \{u + v \geq 2\}\ \cap T_d \bigg) \bigg] \iff \bigg[ z^2 \text{ divides } f_d \bigg].
\end{equation*}
Since $f_d$ is reduced, one endpoint of $L$ must lie inside $u + v \leq 1$. Checking proximity of $L$ to the regions $\{u \leq 1\}$ and $\{v \geq d-1\}$ is similar. By the pigeon-hole principle, one endpoint of $L$ has distance $\leq 1$ to two edges of $T_d$. If these two edges define the upper leftmost corner of $T_d$, then the corresponding endpoint of $L$ arises from one of the following:
\begin{enumerate}[label=(\alph*)]
    \item $A_d = (0, d, 0) \xmapsto{\phi_d} (-d, d)$;
    \item $B_d = (1, d-1, 0) \xmapsto{\phi_d} (-(d-1), d-1)$
    \item $C_d = (1, d-2, 1) \xmapsto{\phi_d} (-(d-2), d-1) $
    \item $D_d = (0, d-1, 1) \xmapsto{\phi_d}(-(d-1), d)$.
\end{enumerate}
If the two edges define a different corner of $T_d$, then after a permutation of $\{x,y,z\}$ we can assume the two edges do in fact define the upper leftmost corner of $T_d$. Note that this process preserves the property of $f_d$ having a diagonal derivation $\delta \in \Der_R(-\log_0 f_d)$. So we will presume we have made this coordinate change.

The other endpoint of $L$ must be have distance $\leq 1$ to the edge $\{u = 0\}$ of $T_d$. We can take this point to come from one of the following applications of $\phi_d$:
\begin{enumerate}[label=(\alph*)]
    \item $\alpha_{d,r} = (r, 0, d-r) \xmapsto{\phi_d} (0, d-r)$ for $0 \leq r \leq d$;
    \item $\beta_{d,r} = (r, 1, d-r-1) \xmapsto{\phi_d} (-1, d-r)$ for $0 \leq r \leq d-1.$
\end{enumerate}

\begin{define} \label{def-lineSegmentNotation}
    Let $p \in \{A_d, B_d, C_d, D_d\}$ and $q \in \{\alpha_{d,r}, \beta_{d,r}\}$ with the legal choices of $r$ described above. Assume $p \neq q$. We slightly abuse notation and denote by $\overline{pq}$ the line segment in $T_d$ connecting $\phi_d(p)$ and $\phi_d(q)$. We use $h_{\overline{pq}}$ to denote a reduced homogeneous polynomial of degree $d$ such that both $\phi_d(\supp(h_{\overline{pq}})) \subset \overline{pq}$ and $p, q \in \phi_d(\supp(h_{\overline{pq}})).$
\end{define}

From the analysis above, we have learned that if $Z \subset \PP^2$ is semi-simple $1$-symmetric, then the cone of $C(Z)$ admits a defining equation that is a monomial (this happens when $L$ is a point) or a defining equation $h_{\overline{pq}}$, where $p \in \{A_d, B_d, C_d, D_d\}$ and $q \in \{\alpha_r, \beta_r\}$. 
 (For $\alpha_{d,r}$ we must have $0 \leq r \leq d$; for $\beta_{d,r}$ we must have $0 \leq r \leq d-1$.) Thus we have given eight families encoding all such $Z$. This is essentially a repeat of the six family enumeration in \cite{PlessisWall3Dim}, though they pass from eight families to six by removing some trivial redundancies. (And at this point their classification into six families is fine enough to sort by Euler characteristic, see loc. cit.; we revisit this classification later.)

Now we extend past \cite{PlessisWall3Dim}. We need the following lemmas.

\begin{lemma} \label{lem-StandardForm-1}
    We have the following factorization of polynomials:
    \begin{itemize}
        \item $h_{\overline{A_d, \beta_{d,r}}} = y h_{\overline{A_{d-1}, \alpha_{d-1,r}}}$;
        \item $h_{\overline{B_d, \beta_{d,r}}} = y h_{\overline{B_{d-1}, \alpha_{d-1,r}}}$;
        \item $h_{\overline{C_d, \beta_{d,r}}} = y h_{\overline{C_{d-1}, \alpha_{d-1,r}}}$;
        \item $h_{\overline{D_d, \beta_{d,r}}} = y h_{\overline{D_{d-1}, \alpha_{d-1,r}}}$.
    \end{itemize}
\end{lemma}

\begin{proof}
    Let $p \in \{A_d, B_d, C_d, D_d\}$. In all cases $y$ divides $h_{\overline{p, \beta_{d,r}}}$. Because $h_{\overline{p, \beta_{d,r}}}$ has a diagonal $\delta \in \Der_R(-\log_0 h_{\overline{p, \beta_{d,r}}})_0$, so does any factor of $h_{\overline{p, \beta_{d,r}}}$, see Remark \ref{rmk-LogDerBasics}.\ref{rmk-item-LogDerDiagonalFactors}.
By Proposition \ref{prop-diagDerLineSement}, $\phi_{d-1}(\supp (\frac{1}{y} h_{\overline{p, \beta_{d,r}}})$ is contained in a line segment inside $T_{d-1}.$ Inspecting all cases leads to the result. For example, $A_{d-1}, \alpha_{d-1, r} \in \supp(\frac{1}{y} h_{\overline{A_d, \beta_{d,r}}})$. Applying $\phi_{d-1}$ to these points must give the endpoints of the line segment containing $\phi_{d-1}(\supp(\frac{1}{y} h_{\overline{A_d, \beta_{d,r}}}))$.
\end{proof}

\begin{lemma} \label{lem-StandardForm-2}
    We have the following equalities, where $c_1, c_2 \in \CC^\star$:
    \begin{itemize}
        \item $h_{\overline{B_d, \alpha_{d,0}}} = c_1 x y^{d-1} + c_2z^d$;
        \item $h_{\overline{C_d, \alpha_{d,d}}} = x(c_1 y^{d-2}z + c_2x^{d-1})$;
        \item $h_{\overline{D_d, \alpha_{d, d}}} = c_1 y^{d-1}z + c_2 x^d.$
    \end{itemize}
\end{lemma}

\begin{proof}
    In all cases, the relevant line segment in $T_d$ contains exactly two points.
\end{proof}

\begin{lemma} \label{lem-StandardForm-3}
    We have the following factorization of polynomials:
    \begin{itemize}
        \item if $r > 0$ then $h_{\overline{B_d \alpha_{d,r}}} = x h_{\overline{A_{d-1}, \alpha_{d-1, r-1}}};$
        \item if $r < d$, then $h_{\overline{C_d \alpha_{d,r}}} = z h_{\overline{B_{d-1}, \alpha_{d-1, r}}}$;
        \item if $r < d$, then $h_{\overline{D_d \alpha_{d,r}}} = x h_{\overline{A_{d-1}, \alpha_{d-1, r}}}$.
    \end{itemize}
\end{lemma}

\begin{proof}
    The divisibility statements follow by inspecting the appropriate line segment in $T_d$. The rest of the argument is entirely similar to Lemma \ref{lem-StandardForm-1}.
\end{proof}

\begin{lemma} \label{lem-decomposableArrangement}
The polynomial  $h_{\overline{A_d, \alpha_{d, 0}}}$ defines a hyperplane arrangement in $\mathbb{C}[y,z]$ and $h_{\overline{A_d, \alpha_{d, d}}}$ defines a hyperplane arrangement in $\mathbb{C}[x,y]$.
\end{lemma}

\begin{proof}
    Using homogeneity, to demonstrate the first claim it suffices to show that $h_{\overline{A_d, \alpha_{d, 0}}} \in \mathbb{C}[y,z]$. The integral points of the line segment $\overline{A_d, \alpha_{d,0}}$ are $\{(-t, d) \mid 0 \leq t \leq d\}$. So $\supp(h_{\overline{A_d, \alpha_{d, 0}}}) \subset \{y^t z^{d-t} \mid 0 \leq t \leq d\}$. The second claim is entirely similar: the integral points of the line segment $\overline{A_d, \alpha_{d,0}}$ are $\{(-t,t) \mid 0 \leq t \leq d\}$ and so $\supp(h_{\overline{A_d, \alpha_{d, d}}}) \subset \{x^{d-t} y^{t} \mid 0 \leq t \leq d\}$.
\end{proof}

\begin{lemma} \label{lem-StandardForm-Aalpha-better}
    Consider a polynomial of the form $h_{\overline{A_d, \alpha_{d,r}}}$ with $1 \leq r \leq d-1$. Let $m$ be the greatest common divisor of $d$ and $r$.
Then there exists $\lambda \in \CC^\star$ such that we have a factorization into irreducible polynomials
    \begin{equation*}
        h_{\overline{A_d, \alpha_{d,r}}} = \prod_{1 \leq q \leq m} h_{\overline{A_{\frac{d}{m}}, \alpha_{\frac{d}{m}, \frac{r}{m}}}} = \lambda \cdot \prod_{1 \leq q \leq m} \left( y^{\frac{d}{m}} + c_q x^{\frac{r}{m}}z^{\frac{d-r}{m}} \right),
    \end{equation*}
    where $c_q \in \CC^\star$ are pairwise distinct.
    Moreover, $h_{\overline{A_d, \alpha_{d,r}}}$ and each of its irreducible factors $h_{\overline{A_{\frac{d}{m}}, \alpha_{\frac{d}{m}, \frac{r}{m}}}}$ has a semi-simple annihilating logarithmic derivation $\frac{r-d}{m} x \partial_x + \frac{r}{m} z \partial_z$.
\end{lemma}

\begin{proof}
    Because there is a nonzero diagonal derivation in $\Der_R(\log_0 h_{\overline{A_d, \alpha_{d,r}}})_0$, every monomial in $\supp(h_{\overline{A_d, \alpha_{d,r}}})$ satisfies the same linear relation. So $\supp(h_{\overline{A_d, \alpha_{d,r}}})$ lies in some plane in $\mathbb{R}^3$, namely the plane spanned by $A_d$ and $\alpha_{d,r}$, regarded as vectors. Encoding the cross product of these vectors as the derivation $d(d-r)x\partial_x - rdz\partial_z$, we deduce this derivation belongs to $\Der_R(-\log_0 h_{\overline{A_d, \alpha_{d,r}}})_0$.

Extracting the greatest common divisor, we find $\delta = \frac{r-d}{m} x \partial_x + \frac{r}{m} z \partial_z \in \Der_R(-\log_0 h_{\overline{A_{d}, \alpha_{d,r}}})_0$, which is manifestly semi-simple. So any $x^a y^b z^c$ in $\supp(h_{\overline{A_d, \alpha_{d,r}}})$ satisfies $\frac{r-d}{m} \cdot a + \frac{r}{m} c = 0$. Since $\frac{r-d}{m}$ and $\frac{r}{m}$ are coprime, $a$ (resp. $c$) is a multiple of $\frac{r}{m}$ (resp. $\frac{r-d}{m}$). We deduce that
    \begin{align*}
        h_{\overline{A_d, \alpha_{d,r}}} = \sum_{0 \leq q \leq m} e_q (y^{\frac{d}{m}})^{m-q} (x^{\frac{r}{m}} z^{\frac{d-r}{m}})^q,
    \end{align*}
    where $e_q \in \CC$ for all $q$ and $e_0, e_m \neq 0.$ Set $\alpha = y^{\frac{d}{m}}$ and $\beta = x^{\frac{r}{m}}z^{\frac{d-r}{m}}$. Regard $h_{\overline{A_d, \alpha_{d,r}}}$ inside $\CC[\alpha, \beta]$. Here it is a homogeneous polynomial in two variables and so factors into linear pieces. After multiplying by $\lambda \in \CC^\star$ to make the coefficient of $y^d \in \supp(h_{\overline{A_d, \alpha_{d,r}}})$ equal one, we deduce that
    \begin{align*}
        h_{\overline{A_d, \alpha_{d,r}}} = \lambda \cdot \prod_{1 \leq q \leq m} \left( y^{\frac{d}{m}} + c_q x^{\frac{r}{m}}z^{\frac{d-r}{m}} \right),
   \end{align*}
    where the $c_q \in \CC^\star$ are distinct. Irreducibility of $y^{\frac{d}{m}} + c_q x^{\frac{r}{m}}z^{\frac{d-r}{m}}$ is equivalent to irreducibility of the dehomogenized $y^{\frac{d}{m}} + c_q x^{\frac{r}{m}} \in \mathbb{C}[x,y]$, which is clear since $\frac{d}{m}$ and $\frac{r}{m}$ are coprime.

That $\delta$ kills each $y^{\frac{d}{m}} + c_q x^{\frac{r}{m}} z^{\frac{d-r}{m}}$ is evident. 
\end{proof}

Consider the following polynomials.

\DefD*

\begin{figure}[h]
\includegraphics[scale=0.25]{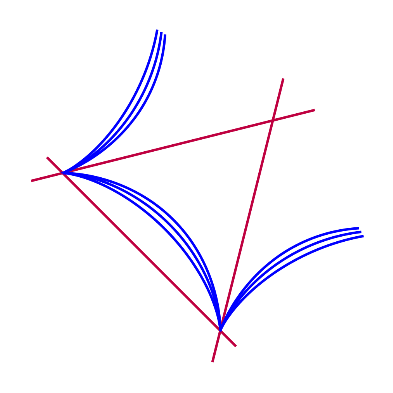}
\caption{A sketch of the projective zero locus of a polynomial in semi-simple standard \emph{provided} $a_1 a_2 a_3 \neq 0$ and $2 \leq u < t$. The coordinate hyperplanes are red, the other components blue, and the intersections are $[0:0:1], [0:1:0], [1:0:0]$. When $u = 1$ (resp. $t-1$) the picture is similar except each blue component is now smooth at $[0:0:1]$ (resp. $[1:0:0]$). When $u = 0$, $f$ defines a decomposable hyperplane arrangement and the picture is different: the blue components are hyperplanes whose only common point is $[1:0:0].$}
\label{fig:specialConfig}
\end{figure}

The next proposition shows the polynomials \eqref{eqn-explicitEnumerationEqn} enumerate all semi-simple $1$-symmetric reduced plane curves.

\begin{proposition} \label{prop-SymmetricReduced-StandardForm-Better}
    Let $Z \subset \PP^2$ be a reduced and semi-simple $1$-symmetric divisor of degree $d$ and $C(Z) \subset \CC^3$ its affine cone. Then, after possibly taking a homogeneous linear change of coordinates on $\CC^3$, we may assume that $C(Z)$ is defined by a reduced homogeneous degree $d$ polynomial $f \in \CC[x,y,z]$ in semi-simple standard form (Definition \ref{def-standardForm}).

    Moreover:
    \begin{enumerate}[label=(\alph*)]
        \item $Z$ is locally quasi-homogeneous and $\Sing(Z) \subseteq \{[1:0:0], [0:1:0], [0:0:1]\}$;
        \item $Z$ is a union of rational curves;
        \item if $Z$ is not a hyperplane arrangement, i.e., if $u > 0$, then
        \begin{equation*}
            \chi(Z) = \begin{cases}
                2 \quad \text{ if } a_1 = a_2 = 0 \\
                3 \quad \text{ otherwise},
            \end{cases}
        \end{equation*}
        where $\chi$ denotes the topological Euler characteristic.
    \end{enumerate}
\end{proposition}

\begin{proof}
    It follows from the discussion in the beginning of this section that for any such $Z$ there exists a homogeneous linear change of coordinates such that $C(Z)$ admits a defining equation of the form $h_{\overline{pq}}$, where $p \in \{A_d, B_d, C_d, D_d\}$ and $q \in \{\alpha_{d,r}, \beta_{d,r}\}$ (where $r$ is legally chosen), see Definition \ref{def-lineSegmentNotation}. Exploiting Lemmata \ref{lem-StandardForm-1}, \ref{lem-StandardForm-2}, \ref{lem-StandardForm-3} we see our defining equation can be taken in the form $x^{e_1} y^{e_2} z^{e_3} h_{\overline{A_{d-e}, \alpha_{d-e, r^\prime}}}$ where $e = e_1+e_2+e_3$ and $r^\prime \leq d-e$. (Performing the calculus of these lemmata leads to some degenerate cases, for example those terminating in Lemma \ref{lem-StandardForm-2}. After another coordinate change these polynomials can be put into the form described above.) If $1 \leq r^\prime \leq d-e - 1$, use Lemma \ref{lem-StandardForm-Aalpha-better} to factor $f$ into irreducibles gives the desired semi-simple standard form. If $r^\prime = 0$ or $r^\prime = d-e$, use Lemma \ref{lem-decomposableArrangement} and (possibly) another coordinate change to produce the the desired form.

    We turn to the `moreover' claims. That $Z$ is locally quasi-homogeneous can be checked directly in all cases. In fact, it was first proven using a coarser enumeration in \cite[pg. 121]{PlessisWall3Dim}. Let $\Gamma_q$ be the projective curve defined by $y^t + c_q x^u z^{t-u}$. Then $\Sing(\Gamma_q) \subseteq \{[1:0:0], [0:0:1]\}$ and, if $q \neq q^\prime$, then $\Gamma_q \cap \Gamma_{q^\prime} = \{[1:0:0], [0:0:1]\}$. The claim about $\Sing(Z)$ easily follows.

    Clearly (b) is true when $u = 0$, i.e., when $Z$ is a hyperplane arrangement. For $u > 0$ it suffices to show that the projective curve $C \subseteq \PP^2$ corresponding to $y^t - x^u z^{t-u} = 0$ is rational. This is straightforward: $a \mapsto (a^t, a^u)$ parametrizes the affine curve $\{y^t - x^u = 0\}$.

    To prove (c), set $u > 0$ and consider the $\CC^\times$-action on $\PP^2$
        \begin{equation*}
            (\CC^\star, \PP^2) \ni (\lambda, [p_1:p_2:p_3]) \mapsto [\lambda^{t-u}p_1: p_2 : \lambda^{-u} p_3] \in \PP^2.
        \end{equation*}
    The fixed points are $\{[1:0:0], [0:1:0], [0:0:1]\}$. And $Z$ is stable under this action. (This is essentially by construction since $(t-u) x \partial_x - u z \partial_z \in \Der_R(-\log f)$.) So $\chi(Z) = \sum \chi(p)$ where the sum runs over the fixed points in $Z$.
\end{proof}

The next proposition explains the non-traceless condition \eqref{eqn-nonTracelessCondition-inIntro}.

\begin{proposition} \label{prop-nonTracelessCondition}
    Let $f \in \CC[x,y,z] \setminus \CC$ be homogeneous with attached divisor $D \subset \PP^2$. Suppose that $D$ is not a normal crossing divisor. Then $f$ admits a semi-simple non-traceless annihilating derivation if and only if, up to a homogeneous, linear coordinate change, $f$ can be written in a semi-simple standard form (Definition \ref{def-standardForm}) satisfying \eqref{eqn-nonTracelessCondition-inIntro}, the non-traceless condition.
\end{proposition}

\begin{proof}
    Suppose that $f$ admits a semi-simple non-traceless annihilating derivation. Proposition \ref{prop-SymmetricReduced-StandardForm-Better} entitles us to change coordinates, represent $f$ as $p$, where $p$ is in semi-simple standard form, and prove the claim for $p$. Set $g = x^{a_1} y^{a_2} z^{a_3}$, $h = \frac{p}{g}$, and $\delta = (u -t) x \partial_x + u z \partial_z$. By inspection $\delta \bullet h = 0$. Define
    \begin{equation*}
        \tau:= \frac{a_1(u-t) + a_3 u}{d} E - \delta \neq 0.
    \end{equation*}
    Then $\tau \in \Der_R(-\log_0 p)_0$ and it is diagonal, hence semi-simple. And $\tau$ is traceless exactly when
    \begin{equation*}
        0 = 3 \bigg( \frac{a_1(u-t) + a_3 u}{d} \bigg) - \bigg( (u-t) + u \bigg) = \frac{1}{d} \bigg((u-t)(3a_1) + u(3a_3) + (u-t)d + ud \bigg).
    \end{equation*}
    So the non-traceless condition is equivalent to the semi-simple derivation $\tau \in \Der_R(-\log_0 p)_0$ being non-traceless.

    The proof will be complete if we show that $\Der_R(-\log_0 p)_0$ has a $\CC$-basis consisting of $\tau$ and possibly some nilpotent derivations.

    \emph{Case 1: $p$ cuts out a hyperplane arrangement.} Note $p$ is a decomposable arrangement: $p = x^{a_1} p^{\prime}$, where $p^\prime =y^{a_2} z^{a_3} h \in \mathbb{C}[y,z]$. So $\Der_R(-\log p) = \Der_R(-\log p_{\red})$ is generated by $x \partial_x$, $y \partial_y + z\partial_z$ and a basis of $\Der_{\mathbb{C}[y,z]}(-\log_0 p_{\red}^\prime)$. The latter is cyclically generated by $(\partial_z \bullet p_{\red}^\prime) \partial_y - (\partial_y \bullet p_{\red}^\prime) \partial_z $, since the partial derivatives of $p_{\red}^\prime$ are a regular sequence. Since $p$ is not normal crossing, $\deg(p^\prime) \geq 3$ and we conclude that $\Der_R(-\log_0 p)_0 =\CC \cdot\tau$.

    \emph{Case 2: $\deg(p_{\red}) \geq 4$.} Proposition \ref{prop-degreeBound} implies that $\Der_R(-\log_0 p) = \mathbb{C} \cdot \tau$.

    \emph{Case 3: $\deg(p_{\red}) = 3$.} We may assume that $p$ is no hyperplane arrangement. So $p_{\red} \in \{y^3 + xz^2, y^3 + x^2z, x(y^2 + xz), y(y^2 + xz), z(y^2+xz)\}$.
     In the first and fourth case direct computation reveals that $\dim_{\CC} (\Der_R(-\log p_{\red})_0) = 2$, and so $\Der_R(-\log_0 p_{\red})_0) = \CC \cdot \tau$ and we are done. In the third case, $\Der_R(-\log_0 p)_0 = \CC \cdot \tau \oplus \CC \cdot \tau^\prime$, where $\tau^\prime = x \partial_y - 2 y \partial_z$ is always nilpotent. The second (resp. fifth) case is symmetric to the first (resp. third).

    \emph{Case 4: $\deg(p_{\red}) = 2$.} We may assume that $p$ is no hyperplane arrangement. So $p = (y^2 + xz)^b$. Computation shows $\Der(-\log_0 p)_0 = \Der(-\log_0 p_{\red})_0$ has three linearly independent generators: $\tau$ and two nilpotent derivations.
\end{proof}

\subsection{Bernstein--Sato polynomials and cones of projective plane curves}

Now we can fulfill our promise of giving all the explicit defining equations of $f \in \CC[x,y,z]$ such that $D_{\red} \subset \PP^2$ has (at worst) quasi-homogeneous (necessarily isolated) singularities and $-3/d$ is not a root of the Bernstein--Sato polynomial of $f$.

\ThmE*

\begin{proof}
    Proposition \ref{prop-nonTracelessCondition} shows (c) is equivalent to $f$ admitting a semi-simple non-traceless annihilating derivation. Now use Theorem \ref{thm-quasiHomIsolatedBSRootCharacterization}.
\end{proof}

\begin{remark} \, \label{rmk-BSFactsSemiSimple}
    \begin{enumerate}[label=(\alph*)]
        \item If $f$ is reduced and in semi-simple standard form, the condition $-3/d \notin \Zeroes(b_D(s))$ is solely determined by the numerics of $t$ and $u$, the possible vanishing of  $a_1, a_2, a_3$, and the number $m$. The non-reduced case is more subtle. 
        \item Suppose that $t=2$ and $u=1$ in Theorem \ref{thm-EnumeratingEqnswithBadBSPoly} so that $D_{\red}$ is a (special) union of lines and conics. Then the non-traceless condition is equivalent to $a_1 \neq a_3$. So when equality holds, $-3/d \in \Zeroes(b_f(s))$.
        \item Suppose $f$ is reduced, in semi-simple standard form, and has no linear factors. Then the non-traceless condition is equivalent to $t \neq 2u$ and (recall $u,t$ are coprime) fails exactly when $t = 2$ and $u = 1$.
        \item Suppose $f = y^t + x^u z^{t-u}$ is in semi-simple standard form with $u \neq 0$. So $1 \leq u < t$ and the pairs $u, t$ and $t-u, t$ are coprime. We claim $-3/t \in \Zeroes(b_f(s))$ if and only if the non-traceless condition fails or $t = 3$. By the above, this is equivalent to saying $[-3/t \in \Zeroes(b_f(s))] \iff [t \in \{2,3\}].$
        
        By an additive Thom--Sebastiani result for Bernstein--Sato polynomials \cite[Proposition 0.7, Theorem 0.8]{SaitoMicrolocalBFunction}, $b_f(s)$ is determined by the data of the Bernstein--Sato polynomials of $y^t$ and $x^u z^{t-u}$. In this case, with $e, \ell_1,\ell_2$ signifying integers,
        \begin{equation*}
            \Zeroes(b_f(s)) = \{-1\} \cup \bigcup_{\substack{1 \leq e \leq t-1 \\ 1 \leq \ell_1 \leq u-1}}\{ -(\frac{e}{t} + \frac{\ell_1}{u}) \} \cup \bigcup_{\substack{1 \leq e \leq t-1 \\ 1 \leq \ell_2 \leq t-u-1}}\{ -(\frac{e}{t} + \frac{\ell_2}{t-u}) \}.
        \end{equation*}
        Suppose $\frac{e}{t} + \frac{\ell_1}{u} = \frac{3}{t}$. Then $(3-e)u = \ell_1t$. Because of our restrictions on $u$ and $t$, this is only possible if $u = 1$ and $t=2$. Symmetric reasoning shows that, if $\frac{e}{t} + \frac{\ell_2}{t-u} = \frac{3}{t}$, then $t-u = 1$ and $t=2$. So indeed $[-3/t \in \Zeroes(b_f(s))] \iff [t \in \{2,3\}]$.
    \end{enumerate}
\end{remark}

\section{Theorem \ref{thm-noBadPole} and Theorem \ref{thm-SMCinourCase}}

This section is mostly a proof of Theorem \ref{thm-noBadPole} which, combined with Theorem \ref{thm-EnumeratingEqnswithBadBSPoly}, quickly implies our main result Theorem \ref{thm-SMCinourCase}. We first recall our target.

\ThmF*

\noindent 
Being in semi-simple standard form, $f$ can be written as
        \begin{align} \label{eqn-explicitEnumerationEqn-inProof}
        f &= x^{a_1}y^{a_2}z^{a_3} \prod_{1 \leq q \leq m} \left( y^t + c_q x^u z^{t-u} \right)^{b_q}.
    \end{align}
Here $m, t, u \in \mathbb{Z}_{> 0}$; $u < t$ and  $u$ and $t$ are coprime; $b_1\in \mathbb{Z}_{> 0}$ and $a_1, a_2, a_3, b_2, \dots, b_m \in \mathbb{Z}_{\geq 0}$;
$c_1, \dots, c_m \in \CC^\times$ are pairwise distinct. Note that $d=a_1 + a_2 + a_3+ t(b_1 + \cdots + b_m)$.
(We have that $1 \leq u < t$ and at least one $b_q$ is nonzero since $f$ is no hyperplane arrangement.)

Also recall that such $f$ satisfies the non-traceless condition when
    \begin{equation} \label{eqn-nonTracelessCondition-in-proof}
        u(d-3a_3) \neq (t-u)(d-3a_1).
\end{equation}

\smallskip The proof of Theorem \ref{thm-noBadPole} is spread out over the following subsections. Unless otherwise stated, we fix $f$ as described in Theorem \ref{thm-noBadPole}.

\subsection{First Reductions}

First, to prove Theorem \ref{thm-noBadPole}, it suffices to prove that $-3/d$ is not a pole of the motivic zeta function (in either the global or local setting). To see this, suppose that $s_0 \neq -3/d$ is some other pole of $Z_f^{\mot}(s)$ of order $m_0 \geq 1$. As always, $D \subset \PP^2$ is the projective divisor attached to $\Div(f)$. By \eqref{eqn-motZetafmlinpieces}, there are two possibilities: $s_0$ is a pole of order $m_0$ of some local zeta function attached to $D$ at some point $p$, where $p$ is in the singular locus of $D_{\red}$; $s_0 = -1/\ell$ is a pole of order $1$, where $\ell$ is the multiplicity of one of the components of $D$. In the former case, \cite[Theorem~D]{Blanco2024} says $s_0$ is a root of multiplicity at least $m_0$ of the corresponding local Bernstein--Sato polynomial; in the latter case, Remark \ref{rmk-BSfacts} demonstrates that the local Bernstein--Sato polynomial of $D$ at a generic point of this component is a multiple of $1+ \ell s$. As all the local Bernstein--Sato polynomials of $D$ divide $b_D(s)$, which in turn divides $b_f(s)$, we conclude that $s_0$ is a root of multiplicity at least $m_0$ of $b_f(s)$. The analogous statement about $Z_{f,0}^{\mot}(s)$ and $b_{f,0}(s)$ is entirely similar: on one hand $b_{f,0}(s) = b_f(s)$; on the other, $Z_{f}^{\mot}(s)$ and $Z_{f,0}^{\mot}(s)$ differ only by a factor  $\LL^{3 + ds}$ by Lemma \ref{lem-GlobalZetaHomogeneousFml}.

Moreover, we will see that our proof that $-3/d$ is not a pole proceeds by computing a residue at $-3/d$ and showing it vanishes. Lemma \ref{lem-GlobalZetaHomogeneousFml} implies that these residues for either the global $Z_f^{\mot}(s)$ or the local $Z_{f,0}^{\mot}(s)$ are the same.

\smallskip So hereafter we will only concern ourselves with showing that $-3/d$ is not a pole of the global motivic zeta function $Z_f^{\mot}(s)$.

\subsection{Set-up}\label{Set-up}
Denote 
$V:= \{xyz \prod_{1 \leq q \leq m} \left( y^t + c_q x^u z^{t-u} \right)=0\}$.  (So $\Div(f)_{\red} \subset V$, and we have equality if and only if $a_1a_2a_3\neq 0$.)

We construct an embedded resolution $h$ of $V\subset \CC^3$ in the obvious way.  First we blow up the origin, creating an exceptional surface $F \cong \PP^2$. Let $V_0$ denote the intersection of $F$ with the strict transform of $V$ (so $V_0$ is given by the same equation as $V$, with now $x,y,z$ homogeneous coordinates on $\PP^2$).
Then we construct the minimal embedded resolution  $\pi:F_0\to F$ of $V_0 \subset F$, denoting by $E_j,j\in T,$ the irreducible components of $\pi^{-1}(V_0)$. For sure all exceptional $E_j$ are isomorphic to $\PP^1$; in our setting this is the case for {\em all} $E_j$ since all components of $V_0$ are rational.

 This induces the embedded resolution $h:Y\to  \CC^3$ of $V$; the irreducible components of $h^{-1}V$ are $F_0$ and an
$\mathbb{A}^1$-bundle $F_j$ over each curve $E_j$.

The numerical data of each component $F_j, j \in \{0\}\cup T,$ are denoted by $(N_j, \nu_j)$; so
\begin{equation}\label{N-definition}
\Div(f \circ h) = dF_0 + \sum_{j \in T} N_jF_j
\end{equation}
and
\begin{equation}\label{nu-definition}
K_Y=\Div(h^*(dx \wedge dy\wedge dz)) = 2 F_0 + \sum_{j \in T} (\nu_j-1)F_j,
\end{equation}
since $(N_0,\nu_0)= (d,3)$. Also, when $F_j$ is a component of the strict transform of $F$, we have that $\nu_j=1$ and that $N_j$ is the corresponding power $a_i$ or $b_q$ in the expression (\ref{eqn-explicitEnumerationEqn-inProof}).


Denoting  $E^\circ_I = \cap_{i \in I} E_i \setminus \cup_{k \notin I} E_k$ for every subset $I$ of $T$, we recall that $F_0$ is the disjoint union of all $E_I^\circ$, $I \subset T$. To simplify notation, we set 
 $E_i^\circ = E_{\{i\}}^\circ$.

\bigskip

\begin{convention} \label{conv:only3chains}
    Figure \ref{fig:EmbeddedResolution} and Figure \ref{fig:indeterminancy} presuppose $u \notin \{1, t-1\}$. Otherwise the pictures are slightly incorrect: if $u = 1$, then $E_x$ intersects $E$ and one of the black `chains' of rational curves is missing; if $u = t-1$, then $E_z$ intersects $E^\prime$ and another black `chain' of rational curves is missing.
\end{convention}

\begin{figure}[h]
\includegraphics[width=\textwidth]{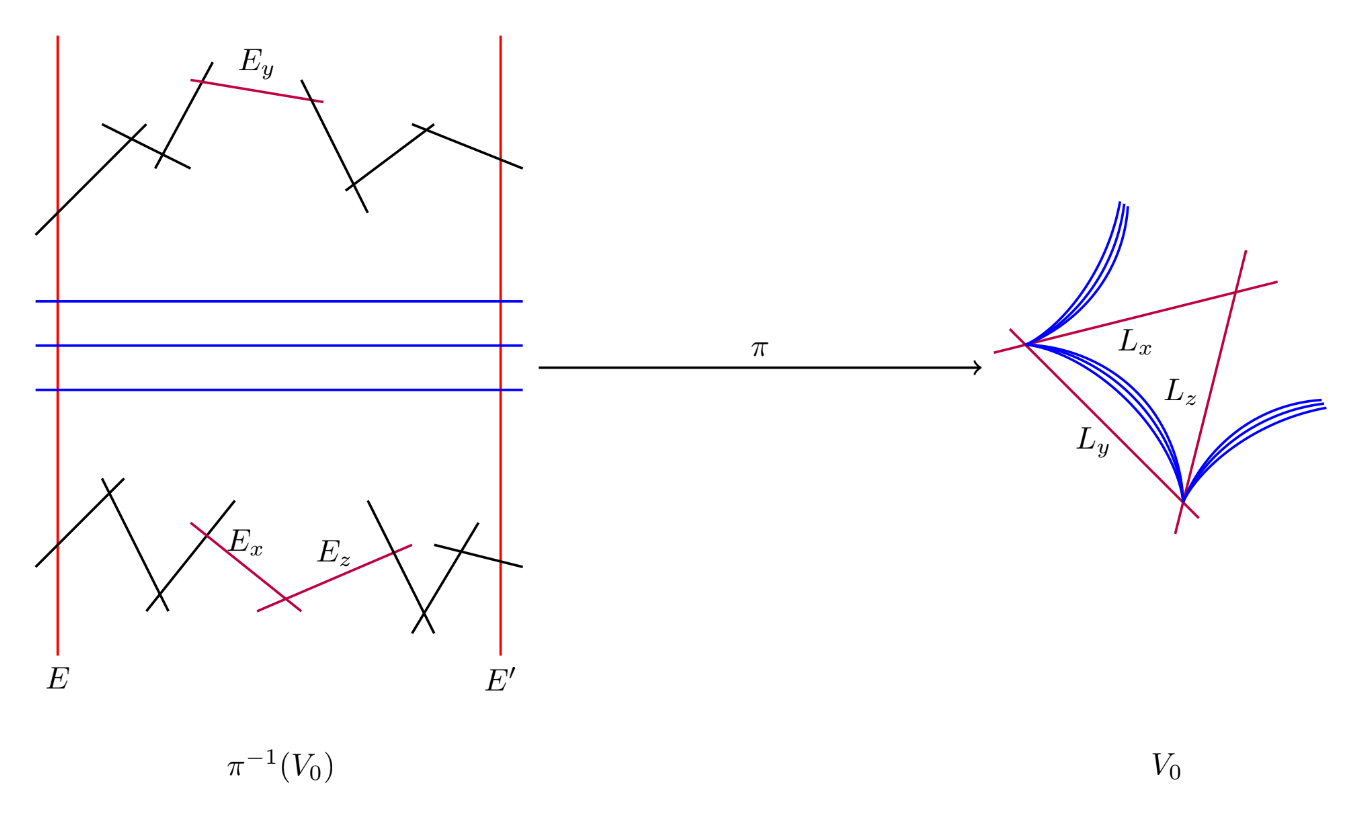}
\centering
\caption{The embedded resolution $\pi: F_0 \to F$ of $V_0 \subset F \simeq \PP^2$ and the components of $\pi^{-1}(V_0)$; $L_x, L_y, L_z$ are the hyperplanes $\{x=0\}, \{y = 0\}, \{z=0\} \subset \PP^2$, and the corresponding strict transforms are $E_x, E_y, E_z$.}
\label{fig:EmbeddedResolution}
\end{figure}

\bigskip
For a homogeneous polynomial, the formula (\ref{zetaformula}) for its motivic zeta function reduces to a formula in terms of an embedded resolution of the associated projective hypersurface, in our case thus in terms of $\pi:F_0\to F\cong \PP^2$. Comparing with Lemma \ref{lem-GlobalZetaHomogeneousFml} yields the following.

\begin{proposition} \label{prop-explicitMotivicFml-Semisimple} Using notation above, we have that
  \begin{align*}
	Z_f^{\mot}(s) &=
\LL^{-3}\cdot \frac{(\LL-1)\LL^{3+ds}}{\LL^{3+ds}-1}  \left( [E_{\emptyset}^\circ] +
\sum_{j\in T} [E_j^\circ] \frac{\LL-1}{\LL^{\nu_j+N_js}-1} \right. \\
&\qquad \qquad \qquad + \left. \sum_{i \neq j\in T} [E_i \cap E_j]\frac{(\LL-1)^2}{(\LL^{\nu_i +N_is}-1)(\LL^{\nu_j +N_js}-1)} \right).
\end{align*}
The local version $Z_{f,0}^{\mot}(s)$ is given by the same formula without the factor $\LL^{3+ds}$.
\end{proposition}
\bigskip

Slightly abusing notation, we rename the two exceptional curves $E_j$ that intersect three times other components in $F_0$ as $E$ and $E'$, as indicated in Figure \ref{fig:EmbeddedResolution}. By a classical calculation we have that $$(N=tu\sum_{i=1}^m b_i + ta_1 + ua_2  ,\ \nu=t+u)$$
 and
$$(N'=t(t-u)\sum_{i=1}^m b_i + ta_3 + (t-u)a_2   ,\ \nu'=2t-u).$$

Direct calculation gives the following birational content to the non-traceless condition.

\begin{lemma} \label{lem-TracelessBirational} The following are equivalent:
    \begin{enumerate}[label=(\alph*)]
        \item the non-traceless condition \eqref{eqn-nonTracelessCondition-in-proof} holds;
        \item $\frac{\nu}{N} \neq \frac{3}{d}$;
        \item $\frac{\nu^\prime}{N^\prime} \neq \frac{3}{d}$.
    \end{enumerate}
\end{lemma}

\bigskip
We now start the proof of Theorem \ref{thm-noBadPole}.
We  treat the general case, meaning that all $a_i>0$.  
Note that then $V_0=D$ and  $\chi(\PP^2 \setminus D_{\red})=0$. The other cases will follow quite immediately from this case (and its proof). We introduce the following notation.

\begin{notation} \label{notation:alpha}
    First, for $C_\ell$ some projective curve, embedded in a surface, we name (the negation of) its self-intersection number on that surface
    \begin{equation*}
        \kappa_\ell = - \deg( C_\ell \cdot C_\ell).
    \end{equation*}
    Second, given some embedded resolution with irreducible components $\{G_k\}, k \in K,$ fix some numerical data $(\tilde{N}, \tilde{\nu})$. Then for $k \in K$ we put
    \begin{equation*}
        \alpha_k = \nu_k - (\tilde{\nu}/\tilde{N})N_k.
    \end{equation*}
    We usually (but not always) use $\alpha$-terms and $\kappa$-terms in the case of our $\{E_j\}, j \in T$. In this case, if the numerical data  $(\tilde{N}, \tilde{\nu})$ are not specified, they are $(d,3)$ and correspond to $F_0$.
Note then that $\alpha_j = 0 \iff \nu_j/N_j =  3/d$, and that `evaluating' $\LL^{\nu_j + s N_j}$ at $s_0 = -3/d$ yields $\LL^{\alpha_j}$.
\end{notation}

Our proof requires manipulating numerical relations among the $\alpha$-terms, and so we will need the following.

\begin{lemma} \label{lem:alphaFacts} \,
\begin{enumerate}[label=(\alph*)]
    \item Denote $\alpha_j := \nu_j - (3/d)N_j$ for $j\in T$. Then $K_{F_0} = \sum_{j\in T} (\alpha_j -1) E_j$.
    \item Fix $E_j, j\in T,$ intersecting two other curves, say $E_{j_1}$ and $E_{j_2}$. Denote $\kappa_j:= -\deg(E_j\cdot E_j)$, as curve on $F_0$.
\begin{enumerate}[label=(\arabic*)]
\item
Then $\kappa_j \alpha_j = \alpha_{j_1} +  \alpha_{j_2}$.
\item
Assume furthermore that $E_j$ is an exceptional curve of $\pi$. Then
$\kappa_j N_j = N_{j_1} +  N_{j_2}$ and $\kappa_j \nu_j = \nu_{j_1} +  \nu_{j_2}$.
\end{enumerate}
\end{enumerate}
\end{lemma}

\begin{proof}
$(a)$
Since the divisor of any rational function is zero in $\Pic(Y)$, the expressions (\ref{nu-definition}) and  (\ref{N-definition}) imply that
\begin{align*}
K_Y&=  2 F_0 + \sum_{j \in T} (\nu_j-1)F_j - \frac3d (dF_0 + \sum_{j \in T} N_jF_j)\\
&= -F_0 + \sum_{j \in T} (\alpha_j-1)F_j
\end{align*}
in $\Pic(Y)\otimes \QQ$. Then the adjunction formula for $F_0\subset Y$ yields  (in $\Pic(F_0)\otimes \QQ$) that
$$
K_{F_0} = (K_Y+F_0)\cdot F_0 = \sum_{j \in T} (\alpha_j-1)F_j \cdot F_0 =
 \sum_{j\in T} (\alpha_j -1) E_j.
$$

$(b)(1)$
Again the adjunction formula, now for $E_j\subset F_0$, yields
$$K_{E_j} = (K_{F_0} + E_j)\cdot E_j = \alpha_j E_j\cdot E_j +  (\alpha_{j_1}-1) E_{j_1}\cdot E_j  +  (\alpha_{j_2}-1) E_{j_2}\cdot E_j,$$
and then $-2=-\kappa_j \alpha_j + (\alpha_{j_1}-1) +  (\alpha_{j_2}-1)$ by taking degrees.

\smallskip
$(b)(2)$ This is well known; see for instance \cite[Lemma 4.1]{VeysMonConjSurvey}.
\end{proof}

Here is a useful consequence.

\begin{lemma} \label{lem:twoVanishingAlphaIntersecting}
    There do not exist two $E_j$ and $E_{j_1}$ with $j, j_1 \in T$ such that $E_j \cap E_{j_1} \neq \emptyset$ and $\alpha_j = 0 = \alpha_{j_1}$.
\end{lemma}

\begin{proof}
    Suppose such $E_j$ and $E_{j_1}$ exists. Call the following components innocent: $E, E^\prime$, and any $E_k$ such that the corresponding $F_k$ is a strict transform of $h$. By Lemma \ref{lem-TracelessBirational} and our assumptions on the $a_i$ and $b_q$, any innocent component has non-vanishing $\alpha$-term. So $E_j$ is not innocent. In particular, $E_j$ intersects exactly two components: $E_{j_1}$ and $E_{j_2}$. Using Lemma \ref{lem:alphaFacts}, we find that $\alpha_{j_2} = 0$ and so $E_{j_2}$ is not innocent. Repeating this argument with $E_{j_2}$ in the role of $E_j$ propagates the vanishing of $\alpha$-terms throughout the components of Figure \ref{fig:EmbeddedResolution} until we eventually reach an innocent component, which is a contradiction.
\end{proof}

\begin{remark} \label{rmk:alphaEvsalphaE'}
    We have $\alpha_E + \alpha_{E^\prime} = 0$. An elementary proof follows by the explicit calculation of $(N, \nu)$ and $(N', \nu^\prime)$, as presented above.
\end{remark}

\subsection{Pole order at most one}

First of all, we claim that $-3/d$ cannot be a pole of order $3$ or $2$ of $Z^{\mot}_f(s)$.  If it is a pole of order $3$, then there must be some $i,i_1 \in T,$ such that $E_i \cap E_{i_1} \neq \emptyset$ and $\alpha_i = 0 = \alpha_{i_1}$. But Lemma \ref{lem:twoVanishingAlphaIntersecting} prohibits this.

So suppose that $-3/d$ is a pole of $Z_f^{\mot}(s)$ of order $2$. Then there must be some $E_i$ with $i \in T$, such that  $\alpha_i = 0$ and $E_i$ intersects exactly two components $E_{i_1}$, $E_{i_2}$. If $E_i$ witnessed $-3/d$ as a pole of order $2$, the following sum must witness $-3/d$ as a pole of order $1$. However,
\begin{align} \label{eqn-ThreeTermResidue-1}
    [E_i^\circ] \frac{\LL-1}{\LL^{\nu_i + N_i s} - 1} &+ [E_{\{i ,i_1\}}^\circ] \frac{\LL-1}{\LL^{\nu_i + N_i s} - 1} \frac{\LL-1}{\LL^{\nu_{i_1} + N_{i_1} s} - 1}
    + [E_{\{i ,i_2\}}^\circ] \frac{\LL-1}{\LL^{\nu_i + N_i s} - 1} \frac{\LL-1}{\LL^{\nu_{i_2} + N_{i_2} s} - 1} \\
    &= \frac{(\LL - 1)^2}{\LL^{\nu_i + N_i s} - 1}\cdot  \bigg( 1 + \frac{1}{\LL^{\nu_{i_1} + N_{i_1}s} - 1} + \frac{1}{\LL^{\nu_{i_2} + N_{i_2}s} - 1} \bigg) \nonumber \\
    &= \frac{(\LL - 1)^2}{\LL^{\nu_i + N_i s} - 1}\cdot  \frac{ \LL^{\nu_{i_1} + \nu_{i_2} + N_{i_1}s + N_{i_2}s} - 1}{(\LL^{\nu_{i_1} + N_{i_2} s} - 1)(\LL^{\nu_{i_1} + N_{i_2} s} - 1)} \nonumber \\
    &= \frac{(\LL - 1)^2}{\LL^{\nu_i + N_i s} - 1}\cdot  \frac{ \LL^{\kappa_i(\nu_i + N_i s)} - 1}{(\LL^{\nu_{i_1} + N_{i_2} s} - 1)(\LL^{\nu_{i_1} + N_{i_2} s} - 1)} \nonumber \\
    &= \frac{(\LL - 1)^2 \sum_{e=0}^{\kappa_i - 1} \LL^{e(\nu_i + N_i s)}}{(\LL^{\nu_{i_1} + N_{i_2} s} - 1)(\LL^{\nu_{i_1} + N_{i_2} s} - 1)}. \nonumber
\end{align}
The penultimate equality employs Lemma \ref{lem:alphaFacts}. By Lemma \ref{lem:twoVanishingAlphaIntersecting}, neither $\alpha_{i_1}$ nor $\alpha_{i_2}$ vanishes. So \eqref{eqn-ThreeTermResidue-1} does not have a pole at $-3/d = -\nu_i/N_i$, and then $-3/d$ is not a pole of $Z_f^{\mot}(s)$ of order two.

\smallskip
The important remark here is that $3/d$ {\em can} be equal to $\nu_i/N_i$ for some $E_i$ in Figure \ref{fig:EmbeddedResolution},  different from $E$ and $E'$.  In that case $-3/d$ is originally a {\em candidate pole of order 2 of $Z^{\mot}_f(s)$}, but, since $E_i$ intersects two other components in Figure \ref{fig:EmbeddedResolution}, we have seen it can finally only be a pole of order one or no pole.


\begin{example}\label{example with alpha zero} 
In fact, the situation above occurs quite frequently, already for reduced $f$. We mention two examples of small degree, one involving the three coordinate hypersurfaces and one with irreducible $f$.

(1) Let $f= xyz(y^3-x^2z)$. Then in  Figure \ref{fig:EmbeddedResolution} both the component  $E_i$ intersecting $E_x$ and the left component $E_i$ intersecting $E_y$ have associated $\nu_i/N_i = 3/d =1/2$.

(2) Let $f= y^5-x^3z^2$. Then in Figure \ref{fig:EmbeddedResolution} the component $E_i$ intersecting $E_x$ has associated $\nu_i/N_i = 3/d =3/5$. (Note that in this case an embedded resolution is visualized by deleting the curves $E_x, E_y,$ and $E_z$ in Figure \ref{fig:EmbeddedResolution}.)

\end{example}

Thus, the general plan is as follows. We consider the sum in the parentheses of Proposition \ref{prop-explicitMotivicFml-Semisimple}, which by \eqref{eqn-ThreeTermResidue-1} we can rewrite as
\begin{equation*} \label{eqn-preResidueFml}
    \sum_{I \subset T, \forall i \in I: \alpha_i \neq 0} [E_I^\circ] \prod_{i \in I} \frac{\LL - 1}{\LL^{\nu_i + N_is} - 1} + \sum_{i \in T, \alpha_i = 0} \frac{(\LL - 1)^2 \sum_{e=0}^{\kappa_i - 1} \LL^{e(\nu_i + N_i s)}}{(\LL^{\nu_{i_1} + N_{i_2} s} - 1)(\LL^{\nu_{i_1} + N_{i_2} s} - 1)},
\end{equation*}
where given any $\alpha_i = 0$ with $i \in T$, $E_{i_1}$ and $E_{i_2}$ are the unique components intersecting $E_i$. We evaluate this expression at $s_0 = -3/d$ to obtain the \emph{residue of $-3/d$ for $Z_f^{\mot}(s)$}:
\begin{equation} \label{eqn-residueFml}
    \sum_{I \subset T, \forall i \in I: \alpha_i \neq 0} [E_I^\circ] \prod_{i \in I} \frac{\LL - 1}{\LL^{\alpha_i} - 1} + \sum_{i \in T, \alpha_i = 0} \kappa_i \frac{(\LL - 1)^2}{(\LL^{\alpha_{i_1}} - 1)(\LL^{\alpha_{i_2}} - 1)}.
\end{equation}
(Technically we have deleted a harmless $\LL^{-3}(\LL - 1)$ factor from the residue.) Our task is to show this residue vanishes, thereby confirming that $-3/d$ is not a pole of $Z_f^{\mot}(s)$.

\medskip
If  {\em all} $\alpha_i \neq 0$, then Theorem \ref{nonpositive euler} below, applied to $F_0$, immediately implies that $-3/d$ is not a pole of $Z^{\mot}_f(s)$.

\begin{theorem}(\cite{VeysCrelle})\label{nonpositive euler}
Let $M$ be a threedimensional non--singular variety and $f: M
\rightarrow \CC$ a non--constant regular function. Let $r : Y
\rightarrow M$ be an embedded resolution of $\{ f = 0 \}$. 
Denote by $F_j,j \in J$, the irreducible components of $r^{-1} \{ f = 0 \}$ and let
$N_j,  \nu_j$ and $F^\circ_J$ be as in the Set-up \ref{Set-up}.
Fix a projective exceptional surface $F_j$ and denote $J_j:= \{i\in J\setminus \{j\} \mid F_i\cap F_j \neq \emptyset\}$.
Suppose that 
$\nu_j/N_j \ne \nu_i/N_i$ (or, equivalently, all $\alpha_i:= \nu_i - (\nu_j/N_j)N_i \neq 0$) for
all $i\in J_j$.

Denote
\begin{equation}\label{residue}
R_{F_j} :=[F^\circ_j] + \sum_{\emptyset \neq I \subset  J_j } [F^\circ_I] \prod_{i \in I} \frac{\LL-1}{\LL^{\alpha_i}-1} ,
\end{equation}
being (up to a nonzero constant factor) `the contribution of $F_j$ to the residue of $-\nu_j/N_j$ for $Z^{\mot}_f(s)$'.

Denote also $C:= \cup_{i\in J_j} (F_i\cap F_j)$.  Assume that $C$ is connected and that $F_j\setminus C$ does not contain a $(-1)$-curve (i.e., a smooth rational projective curve that can be blown down).
If $\chi(F_j^\circ)\leq 0$, then  $R_{F_j} = 0$.
\end{theorem}

\medskip
\subsection{Strategy for vanishing $\alpha$-terms}
Unfortunately, in our setting it is indeed possible that one or more of these $\alpha_i$ vanish (recall Example \ref{example with alpha zero}). 
We assume now that
$\alpha_i = 0$ for at least one $E_i$.

We will use some ideas from the proof of Theorem \ref{nonpositive euler}, and show that its conclusion is still valid here.  We stress that this is not true in general! We will strongly use the concrete graph of the (minimal) embedded resolution of $V_0\subset \PP^2$.  (See Example \ref{conicsexample}, where the conclusion of Theorem \ref{nonpositive euler} is false.)

\bigskip
Consider the rational map $\psi:\PP^2 \dashrightarrow \PP^1=B : [x:y:z] \to [y^t : x^u z^{t-u}]$. The minimal sequence of blow-ups that resolves the indeterminacies of $\psi$ is precisely the minimal embedded resolution map $\pi: F_0\to F\cong \PP^2$.  Let $\tilde\psi:F_0\to B$ denote the resulting morphism.

We have that $\tilde\psi$ maps $E$ and $E'$ bijectively to  $B$. The general fibers of $\tilde\psi$ are the strict transforms in $F_0$ of all curves $y^t + \lambda x^u z^{t-u}, \lambda \in \CC^*$.  The two special fibers, above $0$ and $\infty$, are the chains of $\PP^1$'s in Figure \ref{fig:indeterminancy}, containing $E_y$ and $E_x\cup E_z$, respectively.

The general fibre of $\tilde{\psi}$ being $\PP^1$, it is well known
 that $\tilde\psi$ factors as $p  \circ g$, where $p:\Sigma \to B$ is a ruled surface and $g:F_0\to \Sigma$ is a sequence of blow-downs, contracting the two special fibres to a single $\PP^1$. See Figure \ref{fig:indeterminancy}.

\begin{figure}[h]
\includegraphics[width=\textwidth]{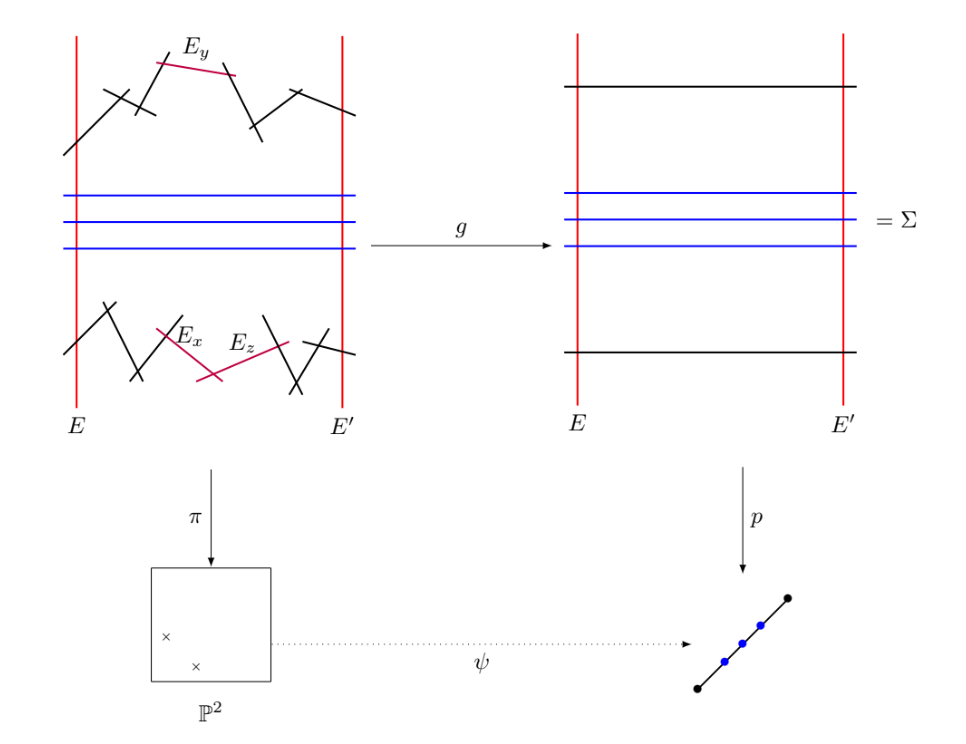}
\centering
\caption{The rational map $\psi: \PP^2 \dashrightarrow \PP^1 = B$ appears in the bottom row; the points labeled $\times$ denote its indeterminancy locus. Here $\pi: F_0 \to \PP^2$ resolves the indeterminancy locus, inducing $\tilde{\psi}: F_0 \to \PP^1$. The blue $\PP^1$'s are some general fibers of $\tilde{\psi}$ which are sent to blue points in $\PP^1$; the chains of $\PP^1$'s are the special fibers of $\tilde{\psi}$, which are sent to the black points of $\PP^1$ (and represent $0$ and $\infty$). Note $E$ and $E^\prime$ biject onto $\PP^1$. Finally, $g: F_0 \to \Sigma$ contracts these chains of $\PP^1$'s to two single $\PP^1$'s yielding the ruled surface $p: \Sigma \to \PP^1$.}
\label{fig:indeterminancy}
\end{figure}

\bigskip
Inspired by the proof of Theorem \ref{nonpositive euler}, our strategy is the following.

\noindent
 (1) We introduce more generally an expression $R$, generalizing (\ref{residue}), to any decorated normal crossing divisor on a non-singular surface with \lq most\rq\ decorations $\alpha_i \neq 0$, such that, in the presence of some $\alpha_i=0$ in our setting, the actual residue of $-3/d$, as candidate pole for $Z^{\mot}_f(s)$, is exactly such an expression.

\noindent
 (2) We show that this expression does not change when performing a blow-up or blow-down \lq within the divisor\rq.

\noindent
 (3) We verify that this expression is zero for the  end configuration in $\Sigma$.

(In fact, realizing all this is not difficult, {\em assuming that all $\alpha_i \neq 0$}).

\medskip
\subsection{Generalized expressions for `residues'}
In \cite{VeysCrelle}, a generalized expression $R$ is introduced, allowing {\em some} $\alpha_i$ to be zero. It will not be a surprise that precisely the term
\begin{equation}\label{zero-contribution}
\kappa_i
\frac{(\LL-1)^2}{(\LL^{\alpha_{i_1}}-1)(\LL^{\alpha_{i_2}} - 1)}.
\end{equation}
will pop up. Indeed, if $E_i$ with $\alpha_i = 0$ exists, it intersects exactly two other components, say $E_{i_1}$ and $E_{i_2}$, and \eqref{zero-contribution} is exactly the contribution of $E_i$ to the residue \eqref{eqn-residueFml}. Here we mention only the (special) case of loc. cit. that we will use.

\begin{define}\label{allowed} \,
\begin{enumerate}[label=(\alph*)]
    \item Let $S$ be a non--singular projective surface, $C=\cup_{i \in T}  C_i$ a simple normal crossing curve on $S$, and $A=\sum_{i \in T} (\beta_i - 1) C_i$ a $\QQ$-divisor on $S$, which is a representative of $K_S$. (Note that, when $\beta_i = 1$, the component $C_i$ is not contained in the support of $A$.)
 We call the pair $(C,A)$ {\em allowed} if for the $i\in T$ with $\beta_i = 0$, we have that
    \begin{enumerate}[label=(\roman*)]
        \item $C_i$ is rational,
        \item exactly two curves $C_{i_1}$ and $C_{i_2}$  intersect $C_i$, and $\beta_{i_1}\neq 0 \neq \beta_{i_2}$.
    \end{enumerate}
\item Let $(C,A)$ be an allowed pair on $S$.
As before, we put $C^\circ_I := (\bigcap_{i \in I} C_i) \setminus
(\bigcup_{\ell \not\in I} C_\ell)$ for $I \subset T$, and denote
$\kappa_i := - \deg(C_i \cdot C_i)$.
We associate to $(C,A)$  the invariant
\begin{equation}\label{R-expression}
R_S                               (C,A) := \sum_{I \subset T, \forall i
\in I : \beta_i \ne 0} [C^\circ_I] \prod_{i \in I}
\frac{\LL-1}{\LL^{\beta_i} - 1} + \sum_{i \in T, \beta_i= 0} 
\kappa_i
\frac{(\LL-1)^2}{(\LL^{\beta_{i_1}}-1)(\LL^{\beta_{i_2}} - 1)}.
\end{equation}
\end{enumerate}
\end{define}

\begin{remark} In (a)(i) above we have that $\beta_{i_1} + \beta_{i_2} = 0$ (as a consequence of the adjunction formula $(K_S + C_i) \cdot C_i  = K_{C_i}$  and the fact that $\beta_i=0$).
\end{remark}


These generalized expressions $R_S(C,A)$ 
turn out to be invariant under the blow-ups that occur in our setting!

\begin{lemma}(\cite[Lemma 5.1]{VeysCrelle})\label{blowup} Let $S$ be a non--singular projective
surface. Let $C=\cup_{i \in T}  C_i$ be a simple normal crossing curve on $S$, and $A=\sum_{\ell \in T} (\beta_\ell - 1)C_\ell$  a representative of $K_S$.
Consider a blow-up $b:\tilde S \to S$ with centre an intersection point of two curves in $C$.
 Then
\begin{enumerate}[label=(\alph*)]
    \item $(C,A)$ is allowed on $S$ if and only if $(b^{-1}C, K_{\tilde S,S} +b^*A)$ is allowed on $\tilde S$;
    \item in that case we have that $R_{\tilde S} (b^{-1} C, K_{\tilde S,S} +b^*A) = R_S (C,A)$.
\end{enumerate}
Here $K_{\tilde S,S}=K_{\tilde S} - b^*K_S$ denotes the relative canonical divisor.
\end{lemma}

\subsection{Theorem \ref{thm-noBadPole}, general case}
\emph{Proof of Theorem \ref{thm-noBadPole} when $a_1 a_2 a_3 \neq 0$.}
    We can now conclude the proof of the general case of Theorem \ref{thm-noBadPole}.  The residue of $-3/d$ is (up to a nonzero constant factor) an expression of the form (\ref{R-expression}) for the curve configuration on $F_0$.
All blow-ups constituting $g$ have a centre that is an intersection of two curves in the curve configuration;  hence by Lemma \ref{blowup} that residue is equal to the expression $R_\Sigma (C,A)$ for the curve configuration on $\Sigma$.

So we must show $R_\Sigma (C,A) = 0$. Keep the names $E$ (resp. $E^\prime$) for our two sections of $\Sigma \to B$; by construction their decorations have not changed, remaining $\alpha$ (resp. $\alpha^\prime$). Use the names $G_1, \dots, G_{m+2}$ for the other components of $C$. Each $G_i$ is a fiber of $\Sigma \to B$, and so $G_i \cdot G_i = 0$. In particular,
\begin{align*}
    R_\Sigma(C, A) = [\Sigma \setminus C] &+ [E^\circ] \frac{\LL - 1}{\LL^\alpha - 1} + [(E')^\circ] \frac{\LL - 1}{\LL^{\alpha^\prime} - 1}
    \\ &+ \sum_{\substack{1 \leq i \leq m+2 \\ \beta_i \neq 0}} \frac{\LL - 1}{\LL^{\beta_i} - 1}  \bigg( [G_i^\circ] + \frac{\LL - 1}{\LL^\alpha - 1} + \frac{\LL - 1}{\LL^{\alpha^{\prime}} - 1} \bigg).
\end{align*}
Since every component of $C$ is rational, $[G_i^\circ] = \LL - 1$ and $[E^\circ] = [(E^\prime)^\circ] = \LL - m - 1$. As $\Sigma \to B \simeq \PP^1$ is a ruled surface, $[\Sigma \setminus (\cup_i G_i)] = [\PP^1] (\LL - m - 1)$ and so $[\Sigma \setminus C] = (\LL - m - 1)(\LL - 1)$. We obtain
\begin{equation*}
    R_{\Sigma}(C,A) = \bigg( 1 + \frac{1}{\LL^\alpha - 1} + \frac{1}{\LL^{\alpha^\prime} - 1} \bigg) \bigg( (\LL - m - 1) (\LL - 1) + \sum_{\substack{1 \leq i \leq m+2 \\ \beta_i \neq 0}} \frac{(\LL - 1)^2}{\LL^{\beta_i} - 1} \bigg).
\end{equation*}
Finally one easily verifies that the first multiplicand, and hence the whole expression, is identically zero, using that $\alpha +\alpha' =0$. \qed


\begin{remark}
    The above reasoning also leads to a more conceptual proof that $\alpha + \alpha^\prime = 0$. The idea is reminiscent of how we found relations among the $\alpha$-terms in Lemma \ref{lem:alphaFacts} using the adjunction formula, except now we use Lemma \ref{blowup} and apply the adjunction formula to the canonical divisor of the ruled surface $\Sigma$.
\end{remark}

\begin{remark} \,
\begin{enumerate}[label=(\alph*)]
    \item In order to prove Theorem \ref{nonpositive euler}, the notion of allowed pair as in Defintion \ref{allowed} has to be extended, allowing curves with $\alpha_i=0$ to intersect any number of curves, see \cite[1.4]{VeysCrelle}.  This more general notion is still \lq almost always\rq\ invariant under blow-ups: there is one exceptional situation where the generalized expression $R_S (C,A)$ changes, more precisely when the centre of blow-up is a point on a component with $\alpha_i=0$, but not on any other component, see \cite[Lemma 1.5]{VeysCrelle}.
    \item The basis of the proof of Theorem \ref{nonpositive euler} is a structure theorem \cite[Theorem 3.5]{VeysStructure} for a normal crossing curve configurations on a rational nonsingular projective surface $S$, with nonpositive Euler characteristic of the complement. The statement is roughly: there exist a sequence of blow-ups $b$ and afterwards a sequence of blow-downs $g$, all \lq within the curve configuration\rq, transforming $S$ into a ruled surface $\Sigma$, with three possible end configurations, including the one on $\Sigma$ above.

    Using this structure theorem, the main point in the proof of Theorem \ref{nonpositive euler} is then showing that the exceptional blow-up/blow-down in (a) does not occur within $b$ or $g$ in the structure theorem. In particular, the fact that {\em all} $\alpha_i \neq 0$ on $S$ is crucially used in the proof that the exceptional case does not occur within $g$!
\end{enumerate}
\end{remark}

\subsection{Theorem \ref{thm-noBadPole}, all cases} \emph{Proof of Theorem \ref{thm-noBadPole}}.
Finally, we treat the special cases where one or more of the $a_i$ are zero, i.e., where not all coordinate hyperplanes are in the support of $f=0$, or, equivalently, not all coordinate lines are in the support of $D$.

Then we consider the same map $\pi : F_0\to F$, which is certainly an embedded resolution of $D$. We consider also the same configuration of curves on $F_0$ as before, that is, we add the missing $E_x$, $E_y$ or $E_z$ to the total inverse image of $D_{\red}$. Associating the value $\alpha_j=1$ to these curves, the residue of $-3/d$ as candidate pole for $Z^{\mot}_f(s)$ is given (up to a nonzero constant factor) by a similar expression $R_S (C,A)$ for the surface $F_0$.  That is, we can view these cases just as special cases of the general case, with some $\alpha$-value(s) equal to $1$. \qed

\begin{example}\label{conicsexample}
Let $f=(yz - x^2)(yz-x^2+y^2)$, then the associated projective curve $D\subset \PP^2\cong F$ consists of two conics, intersecting in one point.
Note thus that the Euler characteristic $\chi (\PP^2 \setminus D) =0$.

 The minimal embedded resolution $\pi:F_0\to F$ of $D$ is a composition of four blow-ups with consecutive exceptional curves $E_1,E_2,E_3,E_4$. See Figure \ref{fig:embedResTwoConics}. In the induced embedded resolution $h:Y \to \CC^3$ of  $\{f=0\}$, the two components of the strict transform have numerical data $(1,1)$, and the exceptional components $F_i$ with data $(N_i,\nu_i)$ are $F_0(4,3)$,
$F_1(2,2)$, $F_2(4,3)$, $F_3(6,4)$, $F_4(8,5)$. Denoting as before $\alpha_i=\nu_i - (\nu_0/N_0) N_i$, we have in particular that $\alpha_1= 1/2$, $\alpha_2=0$, $\alpha_3=-1/2$.

\begin{figure}[h]
\includegraphics[scale=.25]{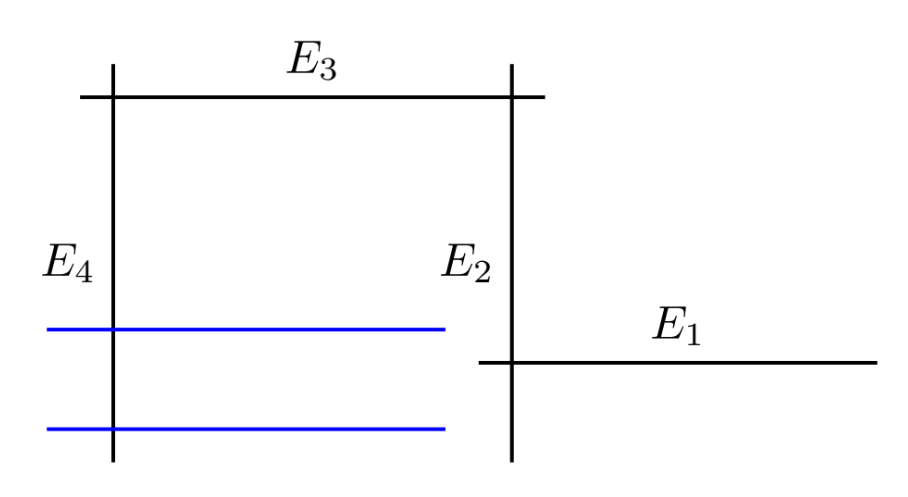}
\centering
\caption{The minimal embedded resolution of the two conics $D \subseteq \PP^2 \simeq F$. The exceptional curves are black with labels $E_1, E_2, E_3, E_4$, whereas the two strict transforms are blue and unlabeled.}
\label{fig:embedResTwoConics}
\end{figure}

So $s_0= -3/d=-3/4$ is a candidate pole of order $2$ for $Z^{\mot}_f(s)$, which by an easy computation turns out to be an effective pole of order $1$. However,  except for the fact that $\alpha_2=0$, all conditions in Theorem \ref{nonpositive euler} are satisfied.
(Note that this conics configuration is a \lq standard example\rq\ in this world; see for instance also  \cite[Example 2.14]{BCLH_MonodromySomeSurface}.)

\end{example}

\begin{remark} The conclusion of Theorem \ref{thm-noBadPole} for the topological zeta function $Z_{f,0}^{\text{top}}(s)$
admits an elementary proof. We sketch this in the case $a_1 a_2 a_3 \neq 0$.

Note first that the analogue of Proposition \ref{prop-explicitMotivicFml-Semisimple} for the topological zeta function is
\begin{align}\label{explicitTopologicalFml-Semisimple}
	&Z_f^{\text{top}}(s) = Z_{f,0}^{\text{top}}(s) \\
&=\frac1{3+ds} \left( \chi(E_{\emptyset}^\circ) +
\sum_{j\in T}  \frac{\chi(E_j^\circ)}{\nu_j+N_js}
 +  \sum_{i \neq j\in T}\frac{\chi(E_i \cap E_j)}{(\nu_i +N_is)(\nu_j +N_js)} \right)\nonumber.
\end{align}
The strategy is to use rather a much more compact formula from \cite{VeysLogCanonical}, where the total contribution of chains of exceptional curves $E_i$, more precisely of the total contribution of each of the four \lq black chains\rq\ in Figure \ref{fig:EmbeddedResolution}, is given by only one simple term. This has in fact a geometric meaning: we can contract each such chain (considered as a curve on $F_0$) to a point, which is then a cyclic quotient singularity on the resulting surface $\bar F$. (That resulting surface is the so-called relative log canonical model of the pair $(F, V)$).  Figure \ref{fig:Qres} shows the images of the remaining components in $\bar F$, and the four quotient singularities, indicated by a black bullet. (Note: when $u =1$ or $u = t-1$, one of the `black chains' is missing. This is no problem as then some chain never got contracted and the `quotient singularity' at the corresponding point is actually smooth.)

In general, say the end curves of a contracted chain $E_\ell, \ell \in C (\subset T),$ intersect the curves $E_i$ and $E_j$, and let $\Delta$ denote the absolute value of the determinant of the intersection matrix of the curves $E_\ell$ in the chain. Then by \cite{VeysLogCanonical} the total contribution of the chain $C$ to  (\ref{explicitTopologicalFml-Semisimple}), namely
\begin{equation*}
    \frac{1}{3+ds} \left( \sum_{\ell \in C} \frac{\chi(E_\ell^\circ)}{v_\ell + N_\ell s} + \sum_{k\in T, \ell \in C, k\neq \ell}\frac{\chi(E_k \cap E_\ell)}{(v_k + N_k s)(v_\ell + N_\ell s)} \right),
\end{equation*}  is
\begin{equation}\label{compact formula}
\frac1{3+ds}\cdot\frac {\Delta}{(\nu_i+N_is)(\nu_j+N_js)}.
\end{equation}

In our setting, these determinants are well known, and indicated on Figure \ref{fig:Qres}.  So the contributions of the points $P_y$, $P_x$, $P'_y$, $P'_z$ are $1/(3+ds)$ times
$$
\frac {t}{(\nu+Ns)(1+a_2s)}, \ \frac {u}{(\nu+Ns)(1+a_1s)}, \ \frac {t}{(\nu'+N's)(1+a_2s)} \ \frac {t-u}{(\nu'+N's)(1+a_3s)},
$$
respectively.

So using \eqref{compact formula} we can simplify \eqref{explicitTopologicalFml-Semisimple}:

\begin{align*} (3+ds)&
Z_f^{\text{top}}(s) =
\frac1{\nu+Ns} \Big(-m+\frac t{1+a_2s}+\frac u{1+a_1s}+\sum_{j=1}^m \frac1{1+b_js} \Big) \\
&+ \frac1{\nu'+N's} \Big(-m+\frac t{1+a_2s}+ \frac {t-u}{1+a_3s}+\sum_{j=1}^m \frac1{1+b_js} \Big) +\frac1{(1+a_1s)(1+a_3s)}.
\end{align*}
Here we have used the specific shape of our embedded resolution diagram: as each $E_j$ is rational, if $E_j \neq E, E^\prime$, then $\chi(E_j^\circ) = 0$ since the Euler characteristic of a twice punctured $\PP^1$ is zero. We have also used the specific shape of $V_0$: $\chi(E_\emptyset^\circ) = 0$ since $a_1 a_2 a_3 \neq 0$.

By our assumptions on the $a_i$ and the $b_j$, it is now clear that $-3/d$ is either a pole of order one or no pole. We now compute the residue of $Z_f^{\text{top}}(s)$ at $-3/d$ by repeatedly utilizing that $\alpha + \alpha^\prime = 0$:
\begin{align*}
&\frac1{\alpha} \Big(-m+\frac t{\alpha_y}+\frac u{\alpha_x}+\sum_{j=1}^m \frac1{\alpha_j} \Big)
+ \frac1{\alpha'} \Big(-m+\frac t{\alpha_y}+ \frac {t-u}{\alpha_z}+\sum_{j=1}^m \frac1{\alpha_j} \Big) +\frac1{\alpha_x\alpha_z} \\
 &=\frac u{\alpha\alpha_x} + \frac{t-u}{\alpha'\alpha_z} +\frac1{\alpha_x\alpha_z}
= \frac{ \alpha_z u -(t-u)\alpha_x + \alpha}{\alpha\alpha_x\alpha_z}.
\end{align*}

Now one checks that this residue vanishes. One option is an elementary computation using explicit expressions for $\alpha_x$, $\alpha_z$, and $\alpha$ (we calculated $(N, \nu)$ previously).

\begin{figure}[h]
\includegraphics[scale=.25]{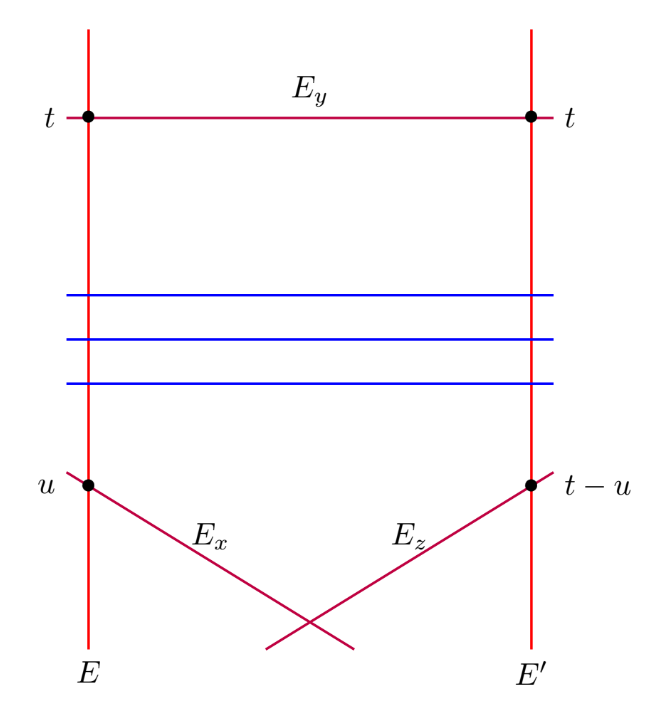}
\centering
\caption{The remaining components on the surface $\tilde{F}$, obtained by contracting the black `chains' in Figure \ref{fig:EmbeddedResolution}. The four black dots represent the cyclic quotient singularities $\{P_y, P_x, P_y^\prime, P_z^\prime\}$, each decorated with the corresponding determinant $\Delta$.}
\label{fig:Qres}
\end{figure}

\bigskip
One could imagine a possible proof for the motivic zeta functions using this strategy.  There is indeed a compact formula for the contribution of the chains above \cite{LMVV} that specializes to (\ref{compact formula}), but with a much more complicated numerator, being a sum of powers of $\LL$ that involve the variable $s$, floor functions and the exact structure of the quotient singularities. For arbitrary $t$ and $u$, it is not clear to us how to show vanishing of the residue this way.

\bigskip
Let us note that, when $a_1 a_2 a_3 \neq 0$, one can also use \cite[Theorem 2.15]{BCLH_MonodromySomeSurface} to conclude Theorem \ref{thm-noBadPole} for the topological zeta function, by checking that $D$ is a `bad' divisor in the language of loc. cit. This requires confirming $-3/d$ is not a pole of the topological zeta function at any point in $D_{\red, \Sing}$, something we (implicitly) concluded using relations among the $\alpha$-terms and the specific shape of our embedded resolution. It also requires $\chi(\PP^2 \setminus D_{\red}) \leq 0$, which depends on which $a_1, a_2, a_3$ vanish. The proof of \cite[Theorem 2.15]{BCLH_MonodromySomeSurface} appeals to a nontrivial classification of pencils of rational curves.

Regardless, there are examples where the motivic zeta function has poles the topological zeta function does not detect. So any proof of Theorem \ref{thm-noBadPole} in the topological setting need not imply Theorem \ref{thm-noBadPole} in the motivic setting.
\end{remark}

\subsection{Theorem \ref{thm-SMCinourCase}}

We deduce our main theorem.

\thmMain*


\begin{proof}
We may assume that $f$ is not normal crossing, since then the claim is straightforward. Now assume that $f$ is a decomposable hyperplane arrangement, i.e., after some homogeneous change of coordinates we may write $f = x^{\ell} f^\prime$ with $f^\prime \in \CC[y,z]$. (This includes the case $f$ is a cone, since we allow $\ell = 0.)$ This product structure implies the candidate poles (of both $Z_f^{\mot}(s)$ and $Z_{f,0}^{\mot}(s)$ are $-\frac{2}{\deg(f^\prime)}$ and $\{-\frac{1}{m_q}\}_q$, where $q$ runs over all the irreducible components of $\{f=0\}$ and $m_q$ counts their multiplicities. By looking at a generic point of each irreducible component, we confirm that $-\frac{1}{m_q}$ is a root of the Bernstein--Sato polynomial $b_f(s)$, see Remark \ref{rmk-BSfacts}. We claim $-\frac{2}{\deg(f^\prime)}$ is also a root of $b_f(s)$. Because of the product structure, $b_{f^\prime}(s)$ divides $b_f(s)$. And Proposition \ref{prop-DiffLinTypeRootCriterion} certifies that $-\frac{2}{\deg(f^\prime)} \in \Zeroes(b_{f^{\prime}}(s))$.

So we may assume that $f$ is not a decomposable hyperplane arrangement. In particular, it is not a cone nor normal crossing. Proposition \ref{prop-SMCCriterion}, along with Theorem \ref{thm-EnumeratingEqnswithBadBSPoly}, then reduces our task to proving the following statement: if $f$ is in semi-simple standard form, is not a hyperplane arrangement, and satisfies the non-traceless condition, and if $-3/d \notin \Zeroes(b_D(s))$, then $-3/d$ is not a pole of $Z_f^{\mot}(s)$ nor $Z_{f,0}^{\mot}(s)$. Note $-3/d \notin \Zeroes(b_D(s))$ implies all $\frac{1}{a_i}$ and $\frac{1}{b_i}$ are different from $\frac{3}{d}$. So we are done by Theorem \ref{thm-noBadPole}.
\end{proof}

\begin{remark} \label{rmk-nonTracelessMystery}
    This manuscript shows that the existence of semi-simple non-traceless logarithmic derivations clearly has content, especially with respect to Bernstein--Sato polynomials, and this is codified by the non-traceless condition \eqref{eqn-nonTracelessCondition-in-proof}, see Proposition \ref{prop-nonTracelessCondition}. However, this non-traceless condition remains mysterious for us. It does have `some' birational content: by Lemma \ref{lem-TracelessBirational}, in our preferred embedded resolution of $V_0$ there are exactly two components $E, E^\prime$ of $\pi^{-1}(V_0)$ that (may) intersect three components, and the non-traceless condition is equivalent to $\alpha_E \neq 0 \neq \alpha_{E^\prime}$. We crucially use this when proving Theorem \ref{thm-noBadPole}.

    It would be very interesting to find a birational interpretation of non-traceless logarithmic derivations in general.
\end{remark}


\bibliographystyle{abbrv}
\bibliography{refs}

\end{document}